\documentclass{amsart}
\pdfoutput=1

\usepackage{amsfonts, amssymb, amsmath, amsthm}                               
\usepackage{slashed}
\usepackage{graphicx}
\usepackage{tikz}                                                             

\newtheorem{theorem}{Theorem} 	      	      	                              
\newtheorem{corollary}[theorem]{Corollary}     	      	      	      	      
\newtheorem{lemma}[theorem]{Lemma}     	       	      	      	      	      
\newtheorem{proposition}[theorem]{Proposition} 	      	      	      	      
\newtheorem{definition}[theorem]{Definition} 	      	      	                
\newtheorem*{remark}{Remark}                                                  

\numberwithin{equation}{section}                                              
\numberwithin{theorem}{section}                                               

\newcommand{\mc}[1]{\mathcal{#1}}                                             

\newcommand{\R}{\mathbb{R}}                                                   
\newcommand{\Sph}{\mathbb{S}}                                                 

\newcommand{\grad}{\nabla^\sharp}                                             
\newcommand{\nasla}{\slashed{\nabla}}                                         

\newcommand{\ui}{\bar{u}}
\newcommand{\vi}{\bar{v}}
\newcommand{\ti}{\bar{t}}
\newcommand{\ri}{\bar{r}}
\newcommand{\ffi}{\bar{f}}
\newcommand{\hhi}{\bar{h}}
\newcommand{\gi}{\bar{g}}

\newcommand{\beq}{\begin{equation}}
\newcommand{\eeq}{\end{equation}}

\allowdisplaybreaks

\begin{document}


\title[Global uniqueness theorems]{Global uniqueness theorems for linear and nonlinear waves}

\author{Spyros Alexakis}
\address{Department of Mathematics\\ 
University of Toronto\\
40 St George Street Rm 6290\\
Toronto, ON M5S 2E4\\ Canada}
\email{alexakis@math.utoronto.ca}

\author{Arick Shao}
\address{Department of Mathematics\\
South Kensington Campus\\
Imperial College London\\
London SW7 2AZ\\ United Kingdom}
\email{c.shao@imperial.ac.uk}


\begin{abstract}
We prove a unique continuation from infinity theorem for regular waves of the form $[ \Box + \mc{V} (t, x) ]\phi=0$.
Under the assumption of no incoming and no outgoing radiation on specific halves of past and future null infinities, we show that the solution must vanish everywhere.
The ``no radiation" assumption is captured in a specific, \emph{finite} rate of decay which in general depends on the $L^\infty$-profile of the potential $\mc{V}$.
We show that the result is optimal in many regards.
These results are then extended to certain power-law type nonlinear wave equations, where the order of decay one must assume is independent of the size of the nonlinear term.
These results are obtained using a new family of global Carleman-type 
estimates on the exterior of a null cone.
A companion paper to this one explores further applications of these new estimates to such nonlinear waves.
\end{abstract}

\setcounter{tocdepth}{2}

\maketitle

\tableofcontents

\section{Introduction} \label{sec:intro}

This paper presents certain global unique continuation results for linear and nonlinear wave equations.
The motivating challenge is to investigate the extent to which globally regular waves can be reconstructed from the radiation they emit towards (suitable portions of) null infinity.
We approach this in the sense of uniqueness: if a regular wave emits no radiation towards appropriate portions of null infinity, then it must vanish.

The belief that a lack of radiation emitted towards infinity should imply the triviality  of the underlying solution has been implicit in the physics literature for many classical fields.
For instance, in the case of linear Maxwell equations, early results in this direction go back at least to \cite{papap:nonradiative}.
Moreover, in general relativity, the question whether non-radiating gravitational fields must be trivial (i.e., stationary) goes back at least to \cite{papap:1957}, in connection with the possibility of time-periodic solutions of Einstein's equations.
The presumption that the answer must be affirmative under suitable assumptions underpins many of the central stipulations in the field; see for example the issue of the final state in \cite{haw_el:gr}. 

We note that the reconstruction of \emph{free} waves in the Minkowski space-time from their radiation fields (for smooth enough initial data) is classically known.
For instance, this has been well-studied in the context of the scattering theory of Lax and Phillips \cite{lax_phil:scatt_theory}.
In odd spatial dimensions, this can be achieved via the explicit representation of the 
 radiation field via the Radon transform of the initial data; in fact, using this method, free waves can be reconstructed from knowledge of their radiation fields on specific halves of future and past null infinities.

Our first goal is to extend this result (in the sense of a uniqueness theorem) to more general wave operators.
We deal here with general, self-adjoint linear wave equations and some nonlinear wave equations over Minkowski spacetime,
\begin{align*}
\mathbb{R}^{n+1} = \{ (t, x) \mid t \in \R \text{, } x = (x^1, \dots, x^n) \in \R^n \} \text{.}
\end{align*}
Our analysis is performed in the exterior region $\mc{D}$ of a light cone.
Roughly, we show that if a solution $\phi$ of such a wave equation decays faster toward null infinity in $\mc{D}$ than the rate enjoyed by free waves (with smooth and rapidly decaying initial data),\footnote{In other words, $\phi$ vanishes at infinity to a given \emph{finite} order.} then $\phi$ itself must vanish on $\mc{D}$.
Our results are essentially global in nature, in that they assume a solution which is smooth on the entire domain $\bar{\mc{D}}$.

\subsubsection*{\nopunct}
\hspace{0.3pc}
In the context of earlier investigations in the physics literature, one generally made sufficient regularity assumptions at infinity to \emph{derive}  vanishing of the 
solution to infinite order at null infinity.
The vanishing of the field can then be derived under the additional assumption of analyticity near a portion of future  null infinity; see, for instance, \cite{bic_scho_tod:time_per_sc:1, bic_scho_tod:time_per_sc:2, papap:1957, papap:1958:1, papap:1958:2}.
However, the assumption of analyticity cannot be justified on physical grounds; in fact, the recent work \cite{io_kl:killing} suggests that the local argument near a piece of null infinity fails without that assumption.

Consequently, it is of interest to ask whether this analyticity condition can be removed in certain settings.
This was accomplished in earlier joint work with V. Schlue, \cite{alex_schl_shao:uc_inf}, where the authors proved that the assumption of vanishing to infinite order at (suitable parts of both) null infinities does imply the vanishing of the solution near null infinity, even without assuming analyticity.
\footnote{In fact, it was shown that the parts of null infinity where one must make this assumption depend strongly on the mass of the background spacetime.}

One can next ask whether the infinite-order vanishing assumption can be further relaxed.
However, straightforward examples in the Minkowski spacetime show that if one does not assume regularity of a wave in a suitably large portion of spacetime, then unique continuation from infinity will fail \emph{unless} one assumes vanishing to infinite order.
Indeed, considering any finite number $k \in \mathbb{N}$ of iterated (spatial) Cartesian derivatives of the Green's function $r^{2-n}$, we obtain a (time-independent) solution of the free wave equation $\Box \phi=0$ which decays towards null infinity at the rate $r^{2-n-k}$ and is regular everywhere, except for the time axis.
 
However, the above still leaves open the question of whether the infinite-order vanishing 
condition can be replaced by vanishing to finite order, if \emph{in addition} one assumes the 
solution to be regular in a suitably large domain which rules out the aforementioned 
counterexamples.
Obviously, one must assume faster decay towards past and future null infinity ($\mc{I}^-$ and $\mc{I}^+$, respectively) than that enjoyed by free linear waves.
The main theorems of this paper derive precisely such results, for self-adjoint linear and a class of nonlinear wave operators over Minkowski spacetime.

\subsubsection*{\nopunct}
\hspace{0.3pc}
Our first result, Theorem \ref{thm.uc_finite}, applies to linear wave equations of the form
\beq
\label{eq} [ \Box + \mc{V} (t, x) ] \phi = 0 \text{.}
\eeq 
Informally speaking, we show that if the potential $\mc{V}$ decays and satisfies suitable $L^\infty$-bounds, and if the solution $\phi$ decays faster on $\mc{D}$ toward null infinity than generic solutions of the free wave equation, then $\phi$ must vanish everywhere on $\mc{D}$.
The required rate of decay for $\phi$ depends on the $L^\infty$-profile of $\mc{V}$---the larger the potentials $\mc{V}$ are assumed to be, the faster we must assume that the solution decays towards infinity.\footnote{The assumed decay, however, remains of finite order.}
Furthermore, we show in Proposition \ref{thm.uc_opt} that this result is essentially optimal for the most general class of decaying potentials $\mc{V}$.
 
A second result, Theorem \ref{thm.uc_strong}, deals with the same equations, but shows that if $\mc{V}$ satisfies a suitable monotonicity property, then it suffices to only assume a specific decay for $\phi$ independent of the $L^\infty$-size of $\mc{V}$.
Moreover, this result generalizes to a class of nonlinear wave equations that includes the usual power-law (defocusing and focusing) nonlinear wave equations; see Theorems \ref{thm.uc_focusing} and \ref{thm.uc_defocusing}.

\begin{figure}
\caption{The null cone $\mc{N}$ and the exterior domain $\mc{D}$ in the Penrose diagram}
\label{fig.penrose}
\includegraphics[width=130pt]{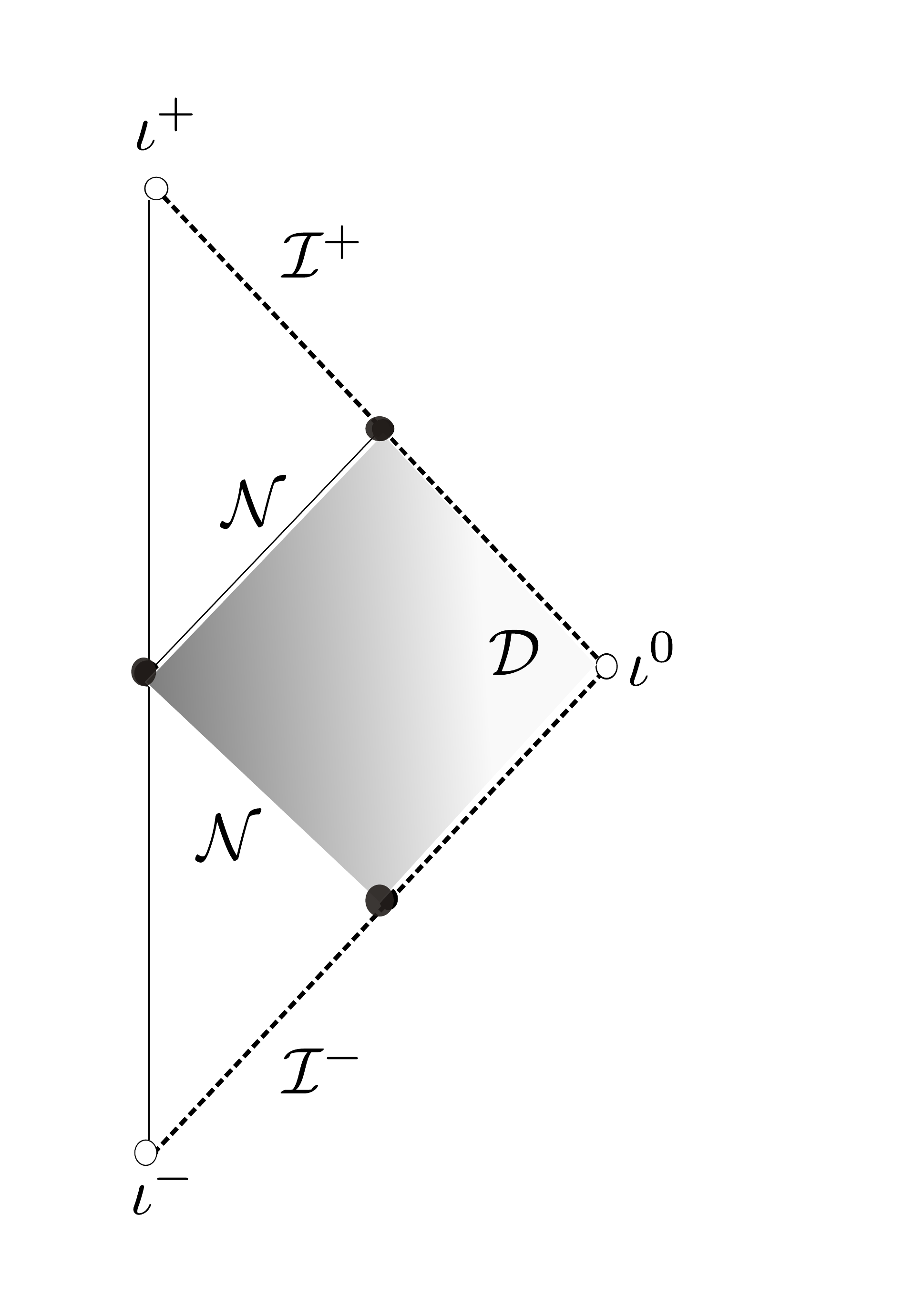}
\end{figure}

The method of proof relies on new 
weighted multiplier estimates applied to a
suitable conjugation of the relevant  equations
 over the entire domain $\mc{D}$.
In a companion paper, \cite{alex_shao:uc_nlw}, we develop these estimates further to derive localized estimates for these classes of wave equations inside time cones.
These estimates are then applied toward understanding the profile of energy concentration near 
singularities.

The precise statements of our unique continuation results are in Section \ref{sec:intro_results}.

\subsubsection*{\nopunct}
\hspace{0.3pc}
As discussed extensively in the introduction of \cite{alex_schl_shao:uc_inf}, both the results here and in \cite{alex_schl_shao:uc_inf} can be considered as strengthenings of some classical results on positive or zero $L^2$-eigenvalues of elliptic operators
\begin{equation} \label{eq.ell} L := - \Delta - \mc{V} \text{,} \end{equation}
which go back to \cite{agmon:lower, kato:growth, mesh:inf_decay, mesh:diff_ineq, sim:pos_eigen}. 
Furthermore, we can also see our results as a strengthening of known work on negative eigenvalues of $L$.

Recall that eigenfunctions of $L$, i.e., $L^2$-solutions of
\begin{equation} \label{eq.eigen} L \psi = \lambda \psi \text{,} \qquad \lambda \in \mathbb{R} \text{,} \end{equation}
yield solutions of the corresponding wave operator
\begin{equation} \label{eq.ell_box} ( \Box + \mc{V} ) \phi = 0 \end{equation}
(with $\mc{V}$ extended to $\R^{n+1}$ as a constant-in-time function).
For positive or zero eigenvalues $\lambda \ge 0$, if $\psi$ solves \eqref{eq.eigen}, then the time-periodic or static solutions
\[ \phi (t, x) := e^{\pm i \sqrt{\lambda} t} \psi(x) \]
solve \eqref{eq.ell_box}.
On the other hand, negative eigenvalues $\lambda < 0$ yield exponentially growing and decaying solutions of \eqref{eq.ell_box} of the form
\[ \phi (t, x)=e^{\pm \sqrt{-\lambda} t} \psi (x) \]

Therefore, any theorem we prove here must incorporate the known non-existence results for positive, zero, and negative eigenvalues of elliptic operators of the form \eqref{eq.ell} for the potentials $\mc{V}$ that we consider here.
Moreover, the known results on the non-existence, or possible decay profile, of such eigenfunctions show how many of the assumptions we are making cannot be relaxed in an essential way.

As shown in the above references, for $\lambda>0$, any corresponding $L^2$-eigenfunction $\psi$ (defined locally near infinity) must in fact vanish to infinite order at infinity.
This then implies (see \cite{hor:lpdo2, hor:lpdo4}) that $\psi=0$ in that neighborhood of infinity.

For the case $\lambda=0$, the same result holds, however the vanishing to infinite order at infinity must be \emph{assumed} and \emph{cannot} be derived.
There exist straightforward examples of solutions $\psi$ to $L \psi = 0$  (for suitable potentials $\mc{V}$) over $\mathbb{R}^n$ which vanish at infinity to any prescribed order $k > 0$, yet $\psi \not\equiv 0$.
(We review such examples in brief in Section \ref{sec:uc_opt}).
Yet, these examples have $\mc{V}$ large (in $L^\infty$) in comparison to the assumed order of vanishing $k$.
In particular, these examples yield nontrivial static solutions to the corresponding wave equation.
These examples thus show that it is \emph{necessary} to assume bounds on the $L^\infty$-profile of the potential $\mc{V}$ which depend on the assumed order of vanishing of the solution at null infinity, in order to conclude the global vanishing of the solution. 

Finally, for the case $\lambda < 0$, we recall that \cite[Theorem 4.2]{mesh:diff_ineq} strengthens a result of Agmon for $L$ with 
$\mc{V} = O(r^{-1-\epsilon})$, where $\epsilon>0$ and $r$ is the Euclidean distance from the origin, to derive that a non-trivial solution $\psi$ of $L \psi = -k^2 \psi$, defined in a neighborhood of spatial infinity, must satisfy
\[ \psi = e^{-k r} r^{-\frac{n - 1}{2}} [ f (\omega) + \varphi(r, \omega) ] \text{,} \]
where
\[ f \in L^2 ( \Sph^{n-1} ) \text{,} \qquad f \not\equiv 0 \text{,} \qquad \int_{ \Sph^{n-1} } | \varphi (r, \omega) |^2 d \omega = O ( r^{-2\gamma} ) \text{,} \qquad \gamma\in (0,\epsilon) \text{.} \]
As discussed above, $\psi(x)$ then defines a solution $\phi (t,x)$ of the wave equation \eqref{eq.ell_box},
\[ \phi (t, x) := e^{\pm k \cdot t} \psi (x) = e^{ k (\pm t - r ) } r^{-\frac{n-1}{2}} [ f ( \omega ) + \varphi(r, \omega) ] \text{.} \]

Note that such solutions will decay exponentially on one of the halves
\[ \mc{I}^+_0 := \{u \le 0 \text{, } v = +\infty \} \text{,} \qquad \mc{I}^-_0 := \{ v \ge 0 \text{, } u = -\infty \} \text{,} \] 
of future and past null infinities $\mc{I}^+, \mc{I}^-$, while having a finite, non-zero radiation field $e^{-2v} f (\omega)$, $e^{-2u} f(\omega)$ on the other.
The condition $f \not\equiv 0$ can precisely be seen as a unique continuation statement: if the eigenfunction $\psi$ gave rise to a wave $\phi$ of vanishing radiation, it would have to vanish itself.
We remark that this discussion also shows that one cannot hope to show uniqueness of solutions 
to wave equations in the form \eqref{eq} for general smooth potentials $\mc{V}$ by even 
assuming infinite order-vanishing on the entire past null infinity.

\subsection{The Main Results} \label{sec:intro_results}

Recall the Minkowski metric on $\R^{n+1}$, given by
\begin{align}
\label{eq.met_mink} g := -dt^2 + dr^2 + r^2 \mathring{\gamma} \text{,} \qquad t \in \R \text{,} \quad r \in [0, \infty) \text{,}
\end{align}
where $r = | x |$, and where $\mathring{\gamma}$ is the round metric on $\Sph^{n-1}$.
In terms of null coordinates,
\begin{align}
\label{eq.null_coord} u := \frac{1}{2} (t - r) \text{,} \qquad v := \frac{1}{2} (t + r) \text{,}
\end{align}
the Minkowski metric takes the form
\begin{align}
\label{eq.met_phys} g = - 4 du dv + r^2 \mathring{\gamma} \text{.}
\end{align}

We use the usual notations $\partial_u, \partial_v$ to refer to derivatives with respect to the $u$-$v$-spherical (i.e., null) coordinates.
For the remaining spherical directions, we use $\nasla$ to denote the induced connections for the level spheres of $(t, r)$.
In particular, we let $| \nasla \phi |^2$ denote the squared ($g$-)norm of the spherical derivatives of $\phi$:
\begin{align}
\label{eq.norm_sph} | \nasla \phi |^2 := g ( \nasla \phi, \nasla \phi ) = r^{-2} \cdot \mathring{\gamma} ( \nasla \phi, \nasla \phi ) \text{.}
\end{align}

Our main results deal with solutions of wave equations in the exterior of the double null cone about the origin,
\begin{align}
\label{eq.outer_region} \mc{D} := \{ Q \in \R^{n+1} \mid | t (Q) | < | r (Q) | \} \text{.}
\end{align}
In our geometric descriptions, we will often refer to the standard Penrose compactification of Minkowski spacetime, see Figure \ref{fig.penrose}, as this provides our basic intuition of the structure of infinity.
With this in mind, we note the following:
\begin{itemize}
\item $\mc{D}$ is the diamond-shaped region in Minkowski spacetime bounded by the null cone about the origin and by the outer half of null infinity $\mc{I}^\pm$.

\item This boundary of $\mc{D}$ has four corners: the origin, spacelike infinity $\iota^0$, and the ``midpoints" of future and past null infinity.
\end{itemize}

\subsubsection{Linear Wave Equations}

Recall if $\phi$ is a free wave, i.e., $\Box \phi \equiv 0$, then the radiation field of $\phi$ at future and past null infinities is captured (up to a constant depending on $u, v$ respectively) by the limit of
\begin{align}
\label{eq.radiation_field_gen} \mc{R} (\phi) := (1 + |u|)^\frac{n-1}{2}  (1 + |v|)^\frac{n-1}{2} \phi \text{.}
\end{align}
This can be seen, for example, via the Penrose compactification by solving the corresponding wave equation on the Einstein cylinder $\R \times \Sph^n$.
We briefly explain the process here; see \cite[Section 11.1]{wald:gr} for details, or \cite{fr:notes_wave} for more general settings.

Invoking the conformal transformation of the Minkowski metric $g$,
\[ \tilde{g} := \Omega^2 g \text{,} \qquad \Omega^2 = 4 ( 1 + u^2 )^{-1} ( 1 + v^2 )^{-1} \text{,} \]
and invoking the conformal covariance of the wave operator on Minkowski, we find that the function $\tilde{\phi} := \Omega^{ - \frac{n-1}{2} } \phi$ solves the linear wave equation
\[ \Box_{ \tilde{g} } \tilde{\phi} - \frac{n - 1}{ 4n } R_{ \tilde{g} } \tilde{\phi} = 0 \text{,} \]
where $R_{ \tilde{g} }$ denotes the (constant) scalar curvature of $\tilde{g}$.
Now, the conformally changed metric $\tilde{g}$ over $\R^{n+1}$ corresponds to a relatively compact domain in $\R \times \Sph^n$ with piecewise smooth boundary---see \cite[Figure 11.1]{wald:gr}.
Null infinities $\mc{I}^+, \mc{I}^-$ correspond to the smooth portions of the boundaries of this domain.

Consequently, when $( \phi|_{t=0}, \partial_t \phi|_{t=0} )$ is sufficiently smooth and decays 
sufficiently rapidly at spatial infinity, one can solve the above wave equation for $\tilde{\phi}$ 
on the entire Einstein cylinder, since the initial data of $\tilde{\phi}$ will be smooth enough on 
$\mathbb{S}^n\times \{0\}$. The asymptotic behavior of $\phi$ at $\mc{I}^\pm$ can be read 
off by restricting $\tilde{\phi}$ to $\mc{I}^\pm$.
In particular, these restrictions correspond (again, up to a constant) to the radiation fields for 
$\phi$ at $\mc{I}^\pm$ defined in \cite{fr:rad_scatt}.

By the same argument, these asymptotics also hold for linear waves with suitably decaying potentials, and suitable nonlinear waves with small initial data; see \cite{chr:null_cond}.

\subsubsection*{\nopunct}
\hspace{0.3pc}
Our first result deals with solutions of linear wave equations of the form
\begin{align}
\label{eq.free_wave_perturb} \Box \phi + \mc{V} \phi = 0 \text{,}
\end{align}
over $\mc{D}$, with general, suitably decaying potentials $\mc{V}$ (with no sign or monotonicity properties), satisfying a bound $\| \mc{V} \|_{ L^\infty ( \mc{D} ) } \leq M$.
We show that when a solution $\phi$ of \eqref{eq.free_wave_perturb} decays towards the outer half of $\mc{I}^\pm$ at a \emph{faster rate}---\emph{how much faster depends in particular on $M$}---than in \eqref{eq.radiation_field_gen}, then $\phi$ must in fact vanish.

\begin{theorem} \label{thm.uc_finite}
Suppose $\phi \in \mc{C}^2 (\mc{D})$ satisfies the differential inequality
\begin{align}
\label{eq.wave_inequality} | \Box \phi | \leq | \mc{V} | | \phi | \text{,} \qquad \mc{V} \in \mc{C}^0 ( \mc{D} ) \text{,}
\end{align}
and suppose $\phi$ satisfies, for some $\beta > 0$, the decay condition
\footnote{The assumption \eqref{eq.wave_decay} can be replaced by corresponding conditions for weighted fluxes through level sets of $f$. This will become apparent from the proof of the theorem.}
\begin{align}
\label{eq.wave_decay} \sup_{ \mc{D} } \left\{ [ ( 1 + |u| ) ( 1 + |v| ) ]^\frac{n - 1 + \beta}{2} ( | u \cdot \partial_u \phi | + | v \cdot \partial_v \phi | ) \right\} < \infty \text{,} \\
\notag \sup_{ \mc{D} \cap \{ |uv| < 1 \} } \left\{ [ ( 1 + |u| ) ( 1 + |v| ) ]^\frac{n - 1 + \beta}{2} | u v |^\frac{1}{2} | \nasla \phi | \right\} < \infty \text{,} \\
\notag \sup_{ \mc{D} } \left\{ [ ( 1 + |u| ) (1 + |v|) ]^\frac{n - 1 + \beta}{2} | \phi | \right\} < \infty \text{.}
\end{align}
Assume in addition that $\mc{V}$ satisfies, for some $0 < p < \beta$, the uniform bound
\begin{equation} \label{eq.uc_potential_decay} | \mc{V} | \leq B p \cdot \min ( \beta - p, p ) \cdot \min ( | uv |^{ -1 + \frac{p}{2} }, | uv |^{ -1 - \frac{p}{2} } ) \text{,} \end{equation}
for some universal constant $B > 0$.
Then, $\phi$ vanishes everywhere on $\mc{D}$.
\end{theorem}

We note that without additional assumptions on $\mc{V}$, the result is essentially optimal.
In regards to the bound \eqref{eq.uc_potential_decay} on $\mc{V}$, this is shown by constructing counterexamples when $\| \mc{V} \|_{ L^\infty (\mc{D}) }$ is too large relative to the rate of decay on $\phi$ at infinity:
 
\begin{proposition} \label{thm.uc_opt}
Given $k > 0$, there exists a sufficiently large $\mc{V} \in \mc{C}^\infty ( \R^{n+1} )$, supported on some region $\{ r \leq R \}$, for which a solution $\phi \in \mc{C}^\infty ( \R^{n+1} )$ to \eqref{eq.free_wave_perturb} exists, is not identically zero on $\mc{D}$, and satifies $| \phi | \lesssim ( 1 + r )^{-k}$.
\end{proposition}

For the proof of Proposition \ref{thm.uc_opt}, see Section \ref{sec:uc_opt}.

\begin{remark}
Note \eqref{eq.wave_inequality} implies $\mc{V}$ decays like $r^{-2-p}$ at spatial infinity, but the condition is weaker---$\mc{V}$ decays like $r^{-1 - \frac{p}{2}}$---at null infinities.
The necessity of assuming $| \mc{V} | \lesssim r^{-2-p}$ is already well-known in the elliptic setting to rule out zero eigenvalues, which correspond to time-independent solutions of \eqref{eq.free_wave_perturb}; see \cite{alex_schl_shao:uc_inf} for a discussion.
\end{remark}

\begin{remark}
It is expected that one \emph{has} to assume vanishing on the portions of $\mc{I}^\pm$ 
that lie in the boundary of $\mc{D}$ for the result to hold.
If one assumed vanishing on a smaller portion of $\mc{I}^\pm$, one expects that the examples in \cite{alin_baou:non_unique} can be modified to produce counterexamples to the vanishing of $\phi$; see also the discussions in \cite{alex_schl_shao:uc_inf}.
\end{remark}

\begin{remark}
We note that if one knows that $\phi$ is $\mc{C}^2$-regular on all of $\R^{1+n}$ and thus vanishes on $\bar{\mc{D}}$, one can derive that $\phi$ vanishes in the entire spacetime by standard energy estimates.
We also note that the assumed decay \eqref{eq.wave_decay} of $\phi$ in $\mc{D}$ is a \emph{quantitative} assumption of no incoming radiation from half of past null infinity, and no outgoing radiation in half of future null infinity. 
\end{remark}

\begin{remark}
We note that a different global uniqueness theorem for operators $\Box+\mc{V}$, with $\mc{V} \in L^{\frac{n+1}{2}}$ and vanishing data across a hyperplane, has been obtained in \cite{ken_ruiz_sog:sobolev_unique}.
\end{remark}

On the other hand, if $\mc{V}$ in \eqref{eq.free_wave_perturb} has a certain sign and monotonicity (and we return to the setting of a differential \emph{equation} rather an inequality), then we only require that $\phi$ vanishes an \emph{arbitrarily small power faster} than in \eqref{eq.radiation_field_gen}.
In particular, the required decay for $\phi$ is no longer tied to the $L^\infty$-size of $\mc{V}$.
 
\begin{theorem} \label{thm.uc_strong}
Suppose $\phi \in \mc{C}^2 (\mc{D})$ satisfies
\begin{align}
\label{eq.wave_strong} \Box \phi + V \phi = 0 \text{,} \qquad V \in \mc{C}^1 ( \mc{D} ) \cap L^\infty ( \mc{D} ) \text{,}
\end{align}
and suppose $\phi$ satisfies, for some (small) $\delta > 0$, the decay condition
\begin{align}
\label{eq.wave_decay_delta} \sup_{ \mc{D} } \left\{ [ ( 1 + |u| ) ( 1 + |v| ) ]^\frac{n - 1 + \delta}{2} ( | u \cdot \partial_u \phi | + | v \cdot \partial_v \phi | ) \right\} < \infty \text{,} \\
\notag \sup_{ \mc{D} \cap \{ |uv| < 1 \} } \left\{ [ ( 1 + |u| ) ( 1 + |v| ) ]^\frac{n - 1 + \delta}{2} | u v |^\frac{1}{2} | \nasla \phi | \right\} < \infty \text{,} \\
\notag \sup_{ \mc{D} } \left\{ [ ( 1 + |u| ) (1 + |v|) ]^\frac{n - 1 + \delta}{2} | \phi | \right\} < \infty \text{.}
\end{align}
as well as
\begin{align}
\label{eq.wave_decay_strong} \sup_{ \mc{D} \cap \{ |uv| > 1 \} } \left\{ [ (1 + |u|) (1 + |v|) ]^\frac{n - 1 + \delta}{2} | uv |^\frac{1}{2} V^\frac{1}{2} | \phi | \right\} < \infty \text{,}
\end{align}
Then, if $V$ satisfies, for some $\mu > 0$, the conditions
\begin{align}
\label{eq.uc_strong_mono} V > 0 \text{,} \qquad ( u \partial_u + v \partial_v ) (\log V) > -2 + \mu \text{.}
\end{align}
then $\phi$ vanishes everywhere on $\mc{D}$.
\end{theorem}

\begin{remark}
Note that $u \partial_u + v \partial_v$ is precisely the dilation vector field on $\R^{n+1}$, which generates a conformal symmetry of Minkowski spacetime.
\end{remark}

\begin{remark}
In particular, Theorem \ref{thm.uc_strong} applies to
\begin{align}
\label{eq.klein_gordon_neg} \Box \phi + \phi = 0 \text{,}
\end{align}
i.e., the negative-mass Klein-Gordon equation.

In contrast, the conclusion of Theorem \ref{thm.uc_strong} fails for the positive-mass Klein-Gordon equation, as its solutions are known to decay at faster rates than for free waves.
In particular, for smooth and rapidly decaying initial data, solutions decay toward null infinity along any geodesic faster than any power of $r^{-1}$; see \cite{kl:remark_kg}.
\end{remark}
 
The above theorems are obtained by weighted multiplier estimates applied to a suitable conjugation of the linear equations; see Section \ref{sec:carleman}.
While Theorem \ref{thm.uc_finite} is proved separately, Theorem \ref{thm.uc_strong} is a special case of uniqueness results for \emph{nonlinear} wave equations, which we discuss below.

\subsubsection{Nonlinear Wave Equations}
 
The aforementioned estimates are sufficiently robust, so that they can be further adapted for certain classes of nonlinear wave equations.
In fact, they yield stronger results for these nonlinear equations, compared to the general linear case. 
The stronger nature of the estimates is manifested in the fact that the order of vanishing at infinity needs only be slightly faster than a specific rate, regardless of the size of the (nonlinear) potential.
  
The class of equations that we will consider will be generalizations of the usual power-law focusing and defocusing nonlinear wave equations:
\footnote{Also applicable is the case $p = 1$: certain linear wave equations with potential.}
\begin{align}
\label{general-NLW} \Box \phi \pm V (t, x) |\phi|^{p-1} \phi = 0 \text{,} \qquad p \geq 1 \text{,}
\end{align}
with $V \in \mc{C}^1 (\mc{D})$ and $V (t, x) > 0$ for all points in $\mc{D}$.
Surprisingly perhaps, it turns out that one obtains these improved estimates in either the focusing or the defocusing case, depending on $p$ as compared to the \emph{conformal} power. 

\begin{definition} \label{foc-defoc}
We call such equations \emph{focusing} if the sign in \eqref{general-NLW} is $+$, and \emph{defocusing} if the sign in \eqref{general-NLW} is $-$.
\end{definition}
    
\begin{definition} \label{subsup-crit}
We refer to \eqref{general-NLW} as:
\begin{itemize}
\item ``Subconformal", if $p < 1 + \frac{4}{n-1}$.

\item ``Conformal", if $p = 1 + \frac{4}{n-1}$.

\item ``Superconformal", if $p > 1 + \frac{4}{n-1}$.
\end{itemize}
\end{definition}

We now state the remaining unique continuation results, one applicable to subconformal focusing-type equations, and the other to both conformal and superconformal defocusing type equations of the form \eqref{general-NLW}.
 
\begin{theorem} \label{thm.uc_focusing}
Suppose $\phi \in \mc{C}^2 (\mc{D})$ satisfies
\begin{align}
\label{eq.wave_focusing} \Box \phi + V | \phi |^{p - 1} \phi = 0 \text{,} \qquad 1 \leq p < 1 + \frac{4}{n - 1} \text{,} \quad V \in \mc{C}^1 ( \mc{D} ) \cap L^\infty ( \mc{D} ) \text{,}
\end{align}
and suppose $\phi$ satisfies, for some $\delta > 0$, the decay condition \eqref{eq.wave_decay_delta}, and
\begin{align}
\label{eq.wave_decay_focusing} \sup_{ \mc{D} \cap \{ |uv| > 1 \} } \left\{ [ (1 + |u|) (1 + |v|) ]^\frac{n - 1 + \beta}{p + 1} | uv |^\frac{1}{p + 1} V^\frac{1}{p + 1} | \phi | \right\} < \infty \text{.}
\end{align}
Then, if $V$ satisfies, for some $\mu > 0$, the conditions
\begin{align}
\label{eq.uc_focusing_mono} V > 0 \text{,} \qquad ( u \partial_u + v \partial_v ) ( \log V ) > - \frac{n - 1}{2} \left( 1 + \frac{4}{n - 1} - p \right) + \mu \text{.}
\end{align}
then $\phi$ vanishes everywhere on $\mc{D}$.
\end{theorem}

\begin{remark}
In particular, taking $p = 1$ in Theorem \ref{thm.uc_focusing} results in Theorem \ref{thm.uc_strong}.
\end{remark}
 
\begin{theorem} \label{thm.uc_defocusing}
Suppose $\phi \in \mc{C}^2 (\mc{D})$ satisfies
\begin{align}
\label{eq.wave_defocusing} \Box \phi - V | \phi |^{p - 1} \phi = 0 \text{,} \qquad p \geq 1 + \frac{4}{n - 1} \text{,} \quad V \in \mc{C}^1 ( \mc{D} ) \cap L^\infty ( \mc{D} ) \text{,}
\end{align}
and suppose $\phi$ satisfies, for some $\delta > 0$, the decay condition \eqref{eq.wave_decay_delta}.
Then, if
\begin{align}
\label{eq.uc_defocusing_mono} V > 0 \text{,} \qquad ( u \partial_u + v \partial_v ) ( \log V ) \leq \frac{n - 1}{2} \left( p - 1 - \frac{4}{n - 1} \right) \text{.}
\end{align}
then $\phi$ vanishes everywhere on $\mc{D}$.
\end{theorem}

\begin{remark}
Note that for defocusing-type equations, \eqref{eq.wave_defocusing}, one does not require the extra decay condition \eqref{eq.wave_decay_focusing} needed for focusing-type equations.
\end{remark}

\begin{remark}
The estimates we employ are robust enough so that Theorems \ref{thm.uc_focusing} and \ref{thm.uc_defocusing} can be even further generalized to operators of the form
\begin{align}
\label{super-general-NLW} \Box \phi \pm V (t, x) \cdot \dot{W} (\phi) = 0 \text{.}
\end{align}
Roughly, we can find analogous unique continuation results in the following cases:
\begin{itemize}
\item Focusing-type, with $\dot{W} (\phi)$ growing at a subconformal rate.

\item Defocusing-type, with $\dot{W} (\phi)$ growing at a conformal or superconformal rate.
\end{itemize}
For simplicity, though, we restrict our attention to equations of the form \eqref{general-NLW}.
\end{remark}

For example, by taking $V (t, x) = 1$, we recover unique continuation results for the usual focusing and defocusing nonlinear wave equations:

\begin{corollary} \label{thm.uc_cor}
Suppose $\phi \in \mc{C}^2 (\mc{D})$ satisfies either
\begin{align}
\label{eq.wave_cor} \begin{cases} \Box \phi + | \phi |^{p - 1} \phi = 0 &\quad 1 \leq p < 1 + \frac{4}{n - 1} \text{,} \\ \Box \phi - | \phi |^{p - 1} \phi = 0 &\quad p \geq 1 + \frac{4}{n - 1} \text{,} \end{cases} 
\end{align}
and suppose \eqref{eq.wave_decay_delta} holds for some $\delta > 0$.
Then, $\phi \equiv 0$ on $\mc{D}$.
\end{corollary}

The main results---Theorems \ref{thm.uc_finite}, \ref{thm.uc_focusing}, and \ref{thm.uc_defocusing}---are proved in Section \ref{sec:uc}.

\subsection{The Main Estimates} \label{sec:intro_carleman}

The main tool for our uniqueness results is a new family of multiplier estimates which are global, in the sense that they apply to regular functions defined over the entire region $\mc{D}$.
The precise estimates---both linear and nonlinear---are presented in Theorems \ref{thm.carleman_lh} and \ref{thm.carleman_nl}.

In many respects, these resemble Carleman-type estimates commonly found in unique continuation results---these are weighted $L^2$-estimates for $\phi$, where the weight has an extra parameter $a \in \R$ that can be freely varied.
On the other hand, in contrast to standard applications, because of the global nature of our estimates, we need not take this parameter $a$ to be large or tending to $\infty$. By slight abuse of language, we refer to our estimates as ``Carleman-type'' from this point onwards.

To derive these estimates, we mostly follow the notations developed in 
\cite{io_kl:unique_ip} and adopted in \cite{alex_schl_shao:uc_inf}.
Similar local Carleman estimates, but from the null cone rather than from null infinity, were proved in \cite{io_kl:private}.
\bigskip

From this geometric point of view, the basic process behind proving Carleman-type estimates can be summarized as follows:
\begin{itemize}
\item One applies multipliers and integrates by parts like for energy estimates, but the goal is now to obtain positive bulk terms.

\item In order to achieve the above, we do not work directly with the solution $\phi$ itself.
Rather, we undergo a conjugation by considering the wave equation for $\psi = e^{-F} \phi$, where $e^{-F}$ is a specially chosen weight function.
\end{itemize}
For further discussion on the geometric view of Carleman estimates and their proofs, the reader is referred to \cite[Section 3.3]{alex_schl_shao:uc_inf}.
 
The function $F$ we work with here is a reparametrization $F(f)$ of the Minkowski square distance function from the origin,
\begin{align}
\label{eq.f} f \in \mc{C}^\infty (\mc{D}) \text{,} \qquad f := -u v = \frac{1}{4} (r^2 - t^2) \text{.}
\end{align}
Its level sets form a family of timelike hyperboloids having zero pseudoconvexity.
As a result of this, the bulk terms we obtain unfortunately yields no first derivative terms.
However, the specific nature of this $f$ crucially helps us, as \emph{it allows us to generate (zero-order) bulk terms that can be made positive on all of $\mc{D}$.}

This global positivity of the bulk is important here primarily because \emph{the weight $e^{-F}$ vanishes on the null cone $\mc{N}$ about the origin}, i.e., the left boundary of $\mc{D}$.
In practice, this allows us to eliminate flux terms at this ``inner'' boundary;
in particular, we do not introduce a cutoff function, which is commonly used in (local) unique continuation problems.
This is ultimately responsible for us only needing to assume finite-order vanishing at null infinity for our uniqueness results.
\footnote{More specifically, in our current context, the lack of a cutoff function removes the need to take $a \nearrow \infty$ in Theorems \ref{thm.carleman_lh} and \ref{thm.carleman_nl} when proving unique continuation.}

The first estimate, Theorem \ref{thm.carleman_lh}, applies to the linear 
wave operator $\Box$.
In this case, special care is required to generate the positive bulk.
In particular, we implicitly utilize that our domain $\mc{D}$ is symmetric up to an 
inversion across a hyperboloid to construct two complementary reparametrizations 
$F_\pm$ which match up at the hyperboloid.
While this conformal inversion (see Section \ref{sec:uc_boundary_inversion}) is not 
used explicitly in the proof, it is manifest in the idea that local Carleman estimates 
from infinity and from a null cone are dual to each other.

For the nonlinear equations that we deal with, our second estimate, Theorem \ref{thm.carleman_nl}, directly uses the nonlinearity to produce a positive bulk.
This is in contrast to the usual method of treating nonlinear terms by seeking to absorb them into the positive bulk arising from the linear terms.
In the context of unique continuation, this results in the improvement (for certain equations) from Theorem \ref{thm.uc_finite}, which requires suitably bounded potentials relative to the assumed order of vanishing, to Theorems \ref{thm.uc_strong}, \ref{thm.uc_focusing}, and \ref{thm.uc_defocusing}, for which no such bound is required.

Finally, we remark that Theorem \ref{thm.carleman_nl} will also be applied in the companion paper \cite{alex_shao:uc_nlw} to study nonlinear wave equations for different purposes.
Thus, we present the main estimates for more general domains than needed in the present paper.

\section{Global Estimates} \label{sec:carleman}

In this section, we derive new Carleman-type estimates for functions $\phi \in \mc{C}^2 (\mc{D})$, where $\mc{D}$ is the exterior of the double null cone (see \eqref{eq.outer_region}),
\begin{align*}
\mc{D} := \{ Q \in \R^{n+1} \mid u (Q) < 0 \text{, } v (Q) > 0 \} \text{.}
\end{align*}
Throughout, we will let $\nabla$ denote the Levi-Civita connection for $( \R^{1+n}, g )$, and we will let $\nabla^\sharp$ denote the gradient operator with respect to $g$.
We also recall the hyperbolic square distance function $f$ defined in \eqref{eq.f}.

\subsection{The Preliminary Estimate} \label{sec:carleman_prelim}

In order to state the upcoming inequalities succinctly, we make some preliminary definitions.

\begin{definition} \label{def.reparam}
We define a \emph{reparametrization} of $f$ to be a function of the form $F \circ f$, where $F \in \mc{C}^\infty (0, \infty)$.
For convenience, we will abbreviate $F \circ f$ by $F$.
We use the symbol $'$ to denote differentiation of a reparametrization as a function on $(0, \infty)$, that is, differentiation with respect to $f$.
\end{definition}

\begin{definition} \label{def.region_adm}
We say that an open, connected subset $\Omega \subseteq \mc{D}$ is \emph{admissible} iff:
\begin{itemize}
\item The closure of $\Omega$ is a compact subset of $\mc{D}$.

\item The boundary $\partial \Omega$ of $\Omega$ is piecewise smooth, with each smooth piece being either a spacelike or a timelike hypersurface of $\mc{D}$.
\end{itemize}
For an admissible $\Omega \subseteq \mc{D}$, we define the \emph{oriented unit normal} $\mc{N}$ of $\partial \Omega$ as follows:
\begin{itemize}
\item $\mc{N}$ is the inward-pointing unit normal on each spacelike piece of $\partial \Omega$.

\item $\mc{N}$ is the outward-pointing unit normal on each timelike piece of $\partial \Omega$.
\end{itemize}
Integrals over such an admissible region $\Omega$ and portions of its boundary $\partial \Omega$ will be with respect to the volume forms induced by $g$.
\end{definition}

\begin{definition} \label{def.monotone_reparam}
We say that a reparametrization $F$ of $f$ is \emph{inward-directed} on an admissible region $\Omega \subseteq \mc{D}$ iff $F' < 0$ everywhere on $\Omega$.
\end{definition}

Lastly, we provide the general form of the wave operators we will consider:
\footnote{These include all the wave operators arising from the main theorems throughout Section \ref{sec:intro_results}.}

\begin{definition} \label{def.wave_op}
Let $U \in \mc{C}^1 (\mc{D} \times \R)$, and let $\dot{U}$ be the partial derivative of $U$ in the last ($\R$-)component.
Define the following (possibly) nonlinear wave operator:
\begin{align}
\label{eq.wave_op} \Box_U \phi (Q) = \Box \phi (Q) + \dot{U} (Q, \phi (Q) ) \text{,} \qquad Q \in \mc{D} \text{.}
\end{align}
Furthermore, derivatives $\nabla U$ of $U$ will be with respect to the first ($\mc{D}$-)component.
\end{definition}

We can now state our preliminary Carleman-type inequality:

\begin{proposition} \label{thm.carleman_gen}
Let $\phi \in \mc{C}^2 (\mc{D})$, and let $\Omega \subseteq \mc{D}$ be an 
admissible region.
Furthermore, let $U$ and $\Box_U$ be as in Definition \ref{def.wave_op}.
Then, for any inward-directed reparametrization $F$ of $f$ on $\Omega$, the 
following inequality holds,
\begin{align}
\label{eq.carleman_gen} \frac{1}{8} \int_\Omega e^{-2 F} | F^\prime |^{-1} \cdot | \Box_U \phi |^2 &\geq \int_\Omega e^{-2 F} ( f | F^\prime | G_F - H_F ) \cdot \phi^2 \\
\notag &\qquad - \int_\Omega \mc{B}_U^F - \int_{ \partial \Omega } P^F_\beta \mc{N}^\beta \text{,}
\end{align}
where:
\begin{itemize}
\item $G_F$ and $H_F$ are defined in terms of $f$ as
\begin{align}
\label{eq.G} G_F := - ( f F' )' \text{,} \qquad H_F := \frac{1}{2} ( f G_F )' \text{,}
\end{align}

\item $\mc{B}_U^F$ is the nonlinear bulk quantity,
\begin{align}
\label{eq.bulk} \mc{B}_U^F &:= e^{-2 F} \left( \frac{n - 1}{4} - f F' \right) \cdot \dot{U} (\phi) \phi - e^{-2 F} \nabla^\alpha f \cdot \nabla_\alpha U (\phi) \\
\notag &\qquad - 2 e^{-2 F} \left( \frac{n + 1}{4} - f F' \right) \cdot U (\phi) \text{,}
\end{align}

\item $\mc{N}$ is the oriented unit normal of $\partial \Omega$.

\item $P^F$ is the current,
\begin{align}
\label{eq.current} P^F_\beta &:= e^{-2 F} \left( \nabla^\alpha f \cdot \nabla_\alpha \phi \nabla_\beta \phi - \frac{1}{2} \nabla_\beta f \cdot \nabla^\mu \phi \nabla_\mu \phi \right) \\
\notag &\qquad + e^{-2 F} \nabla_\beta f \cdot U (\phi) + e^{-2 F} \left( \frac{n - 1}{4} - f F' \right) \cdot \phi \nabla_\beta \phi \\
\notag &\qquad + e^{-2 F} \left[ \left( f F' - \frac{n - 1}{4} \right) F' - \frac{1}{2} G_F \right] \nabla_\beta f \cdot \phi^2 \text{.}
\end{align}
\end{itemize}
\end{proposition}

The remainder of this subsection is dedicated to the proof of \eqref{eq.carleman_gen}.

\subsubsection{Preliminaries}

We first collect some elementary computations regarding $f$.

\begin{lemma} \label{thm.f_deriv}
The following identities hold:
\begin{align}
\label{eq.f_deriv} \partial_u f = -v \text{,} \qquad \partial_v f = -u \text{.}
\end{align}
As a result,
\begin{align}
\label{eq.f_grad} \grad f = \frac{1}{2} ( u \cdot \partial_u + v \cdot \partial_v ) \text{,} \qquad \nabla^\alpha f \nabla_\alpha f = f \text{.}
\end{align}
\end{lemma}

\begin{lemma} \label{thm.f_hessian}
The following identity holds:
\begin{align}
\label{eq.f_hessian} \nabla^2 f = \frac{1}{2} g \text{.}
\end{align}
Moreover,
\begin{align}
\label{eq.f_hessian_ex} \Box f = \frac{n+1}{2} \text{,} \qquad \nabla^\alpha f \nabla^\beta f \nabla_{\alpha\beta} f = \frac{1}{2} f \text{.}
\end{align}
\end{lemma}

\begin{remark}
In particular, equation \eqref{eq.f_hessian} implies that the level sets of $f$ have exactly zero pseudoconvexity.
\end{remark}

We also recall the (Lorentzian) divergence theorem in terms of our current language: \emph{if $\Omega \subseteq \mc{D}$ be admissible, and if $P$ is a smooth $1$-form on $\mc{D}$, then}
\begin{align}
\label{eq.divg_thm} \int_\Omega \nabla^\beta P_\beta &= \int_{ \partial \Omega } P_\beta \mc{N}^\beta \text{,}
\end{align}
\emph{where $\mc{N}$ is the oriented unit normal of $\partial \Omega$.}

\subsubsection{Proof of Proposition \ref{thm.carleman_gen}}

We begin by defining the following shorthands:
\begin{itemize}
\item Define $\psi \in \mc{C}^\infty (\mc{D})$ and the conjugated operator $\mc{L}_U$ by
\begin{align}
\label{eq.psi_phi} \psi := e^{-F} \phi \text{,} \qquad \mc{L}_U \psi := e^{-F} \Box_U \phi = e^{-F} \Box_U ( e^F \psi ) \text{.}
\end{align}

\item Let $S$ and $S_\ast$ define the operators
\begin{align}
\label{eq.S} S \psi := \nabla^\alpha f \nabla_\alpha \psi \text{,} \qquad S_\ast \psi := S \psi + \frac{n-1}{4} \cdot \psi \text{.}
\end{align}

\item Recall the stress-energy tensor for the wave equation, applied to $\psi$:
\begin{align}
\label{eq.stress_energy} \mc{Q}_{\alpha\beta} [\psi] := \mc{Q}_{\alpha\beta} := \nabla_\alpha \psi \nabla_\beta \psi - \frac{1}{2} g_{\alpha\beta} \nabla^\mu \psi \nabla_\mu \psi \text{.}
\end{align}
\end{itemize}
The proof will revolve around an energy estimate for the wave equation, but for $\psi$ rather than $\phi$.
We also make note of the following relations between $\psi$ and $\phi$:

\begin{lemma} \label{thm.algebraic_pre}
The following identities hold:
\begin{align}
\label{eq.deriv_psi_phi} \nabla_\alpha \psi = e^{-F} ( \nabla_\alpha \phi - F' \nabla_\alpha f \cdot \phi ) \text{,} \qquad S \psi = e^{-F} ( S \phi - f F' \cdot \phi ) \text{.}
\end{align}
Furthermore, we have the expansion
\begin{align}
\label{eq.L_psi_phi} \mc{L}_U \psi &= \Box \psi + 2 F^\prime \cdot S_\ast \psi + [ f ( F^\prime )^2 - G_F ] \cdot \psi + e^{-F} \dot{U} (\phi) \text{.}
\end{align}
\end{lemma}

\begin{proof}
First, \eqref{eq.deriv_psi_phi} is immediate from definition and from \eqref{eq.f_grad}.
We next compute
\begin{align}
\label{eql.L_psi_phi_0} \mc{L}_U \psi &= e^{-F} \nabla^\alpha ( F^\prime e^F \nabla_\alpha f \cdot \psi ) + e^{-F} \nabla^\alpha ( e^F \nabla_\alpha \psi ) + e^{-F} \dot{U} (\phi) \\
\notag &= \Box \psi + 2 F^\prime \cdot S \psi + f ( F^\prime )^2 \cdot \psi + f F^{\prime\prime} \cdot \psi + F^\prime \Box f \cdot \psi + e^{-F} \dot{U} (\phi) \text{,}
\end{align}
where we again applied \eqref{eq.f_grad}.
Since \eqref{eq.G} and \eqref{eq.f_hessian_ex} imply
\begin{align}
\label{eql.L_psi_phi_1} f ( F^\prime )^2 + f F^{\prime\prime} + F^\prime \Box f &= f ( F^\prime )^2 + ( f F^{\prime\prime} + F^\prime ) + \frac{n - 1}{2} F^\prime \\
\notag &= f ( F^\prime )^2 - G_F + \frac{n - 1}{2} F^\prime \text{,}
\end{align}
then \eqref{eq.L_psi_phi} follows from applying \eqref{eql.L_psi_phi_1} to \eqref{eql.L_psi_phi_0}.
\end{proof}

The first step in proving Proposition \ref{thm.carleman_gen} is to expand $\mc{L}_U \psi S_\ast \psi$:

\begin{lemma} \label{thm.algebraic}
The following identity holds,
\begin{align}
\label{eq.algebraic} \mc{L}_U \psi S_\ast \psi &= 2 F^\prime \cdot | S_\ast \psi |^2 + ( f F^\prime G_F + H_F ) \cdot \psi^2 + \mc{B}^F_U + \nabla^\beta P^F_\beta \text{,}
\end{align}
\end{lemma}

\begin{proof}
From the stress-energy tensor \eqref{eq.stress_energy}, we can compute,
\begin{align}
\label{eql.algebraic_00} \nabla^\beta ( \mc{Q}_{\alpha\beta} \nabla^\alpha f ) &= \nabla^\beta \mc{Q}_{\alpha\beta} \nabla^\alpha f + \mc{Q}_{\alpha\beta} \nabla^{\alpha\beta} f \\
\notag &= \Box \psi S \psi + \nabla^2_{\alpha\beta} f \cdot \nabla^\alpha \psi \nabla^\beta \psi - \frac{1}{2} \Box f \cdot \nabla^\beta \psi \nabla_\beta \psi \text{,} \\
\notag \nabla^\beta ( \psi \nabla_\beta \psi ) &= \psi \Box \psi + \nabla^\beta \psi \nabla_\beta \psi \text{.}
\end{align}
Summing the equations in \eqref{eql.algebraic_00} and recalling \eqref{eq.f_hessian} and \eqref{eq.f_hessian_ex}, we obtain
\begin{align}
\label{eql.algebraic_01} \nabla^\beta \left( \mc{Q}_{\alpha\beta} \nabla^\alpha f + \frac{n-1}{4} \cdot \psi \nabla_\beta \psi \right) &= \Box \psi S_\ast \psi \text{.}
\end{align}
Multiplying \eqref{eq.L_psi_phi} by $S_\ast \psi$ and applying \eqref{eql.algebraic_01} results in the identity
\begin{align}
\label{eql.algebraic_0} \mc{L}_U \psi S_\ast \psi &= 2 F^\prime \cdot | S_\ast \psi |^2 + [ f (F^\prime)^2 - G_F ] \cdot \psi S_\ast \psi + e^{-F} \dot{U} (\phi) S_\ast \psi \\
\notag &\qquad + \nabla^\beta \left( \mc{Q}_{\alpha\beta} \nabla^\alpha f + \frac{n-1}{4} \cdot \psi \nabla_\beta \psi \right) \text{.}
\end{align}

Next, letting $\mc{A} = f (F^\prime)^2 - G_F$, the product rule and \eqref{eq.f_hessian_ex} imply
\begin{align}
\label{eql.algebraic_1} \mc{A} \cdot \psi S_\ast \psi &= \frac{1}{2} \mc{A} \cdot \nabla^\beta f \nabla_\beta ( \psi^2 ) + \frac{n-1}{4} \mc{A} \cdot \psi^2 \\
\notag &= \frac{1}{2} \nabla^\beta ( \mc{A} \nabla_\beta f \cdot \psi^2 ) - \frac{1}{2} \nabla^\beta f \nabla_\beta \mc{A} \cdot \psi^2 - \frac{1}{2} \mc{A} \cdot \psi^2 \\
\notag &= \frac{1}{2} \nabla^\beta ( \mc{A} \nabla_\beta f \cdot \psi^2 ) - \frac{1}{2} (f \mc{A})^\prime \cdot \psi^2 \\
\notag &= \frac{1}{2} \nabla^\beta ( \mc{A} \nabla_\beta f \cdot \psi^2 ) + ( f F^\prime G_F + H_F ) \cdot \psi^2 \text{.}
\end{align}
Moreover, recalling \eqref{eq.deriv_psi_phi}, we can write
\begin{align}
\label{eql.algebraic_20} e^{-F} \dot{U} (\phi) S_\ast \psi &= e^{-2 F} \dot{U} (\phi) S \phi + e^{-2 F} \left( \frac{n - 1}{4} - f F' \right) \dot{U} (\phi) \phi \text{.}
\end{align}
From the product and chain rules, \eqref{eq.f_grad}, and \eqref{eq.f_hessian_ex}, we see that
\begin{align}
\label{eql.algebraic_21} e^{-2 F} \cdot \dot{U} (\phi) S \phi &= \nabla^\beta [ e^{-2 F} \nabla_\beta f \cdot U (\phi) ] - e^{-2 F} \cdot S U (\phi) \\
\notag &\qquad - \nabla^\beta ( e^{-2 F} \nabla_\beta f ) \cdot U (\phi) \\
\notag &= \nabla^\beta [ e^{-2 F} \nabla_\beta f \cdot U (\phi) ] - e^{-2 F} \cdot S U (\phi) \\
\notag &\qquad - 2 e^{-2 F} \left( \frac{n + 1}{4} - f F' \right) \cdot U (\phi) \text{,}
\end{align}
where $S U := \nabla^\alpha f \nabla_\alpha U$ is a derivative of $U$ only in the first ($\mc{D}$-)component.
Therefore, from \eqref{eql.algebraic_20} and \eqref{eql.algebraic_21}, it follows that
\begin{align}
\label{eql.algebraic_2} e^{-F} \dot{U} (\phi) S_\ast \psi &= \nabla^\beta [ e^{-2 F} \nabla_\beta f \cdot U (\phi) ] + \mc{B}^F_U \text{.}
\end{align}

Combining \eqref{eql.algebraic_0} with \eqref{eql.algebraic_1} and \eqref{eql.algebraic_2} yields
\begin{align}
\label{eql.algebraic_3} \mc{L}_U \psi S_\ast \psi &= 2 F^\prime \cdot | S_\ast \psi |^2 + ( f F^\prime G_F + H_F ) \cdot \psi^2 + \mc{B}_U^F \\
\notag &\qquad + \nabla^\beta \left[ e^{-2 F} \nabla_\beta f \cdot U (\phi) + \frac{1}{2} \mc{A} \nabla_\beta f \cdot \psi^2 \right] \\
\notag &\qquad + \nabla^\beta \left( \mc{Q}_{\alpha\beta} \nabla^\alpha f + \frac{n-1}{4} \cdot \psi \nabla_\beta \psi \right)
\end{align}
Thus, to prove \eqref{eq.algebraic}, it remains only to show that
\begin{align}
\label{eql.algebraic_4} P^F_\beta = \mc{Q}_{\alpha\beta} \nabla^\alpha f + \frac{n-1}{4} \cdot \psi \nabla_\beta \psi + \frac{1}{2} \mc{A} \nabla_\beta f \cdot \psi^2 + e^{-2 F} \nabla_\beta f \cdot U (\phi) \text{.}
\end{align}

Note we obtain from \eqref{eq.f_grad}, \eqref{eq.stress_energy}, and \eqref{eq.deriv_psi_phi} that
\begin{align}
\label{eql.algebraic_41} \mc{Q}_{\alpha\beta} \nabla^\alpha f &= e^{-2 F} \nabla^\alpha f ( \nabla_\alpha \phi - F' \nabla_\alpha f \cdot \phi ) ( \nabla_\beta \phi - F' \nabla_\beta f \cdot \phi ) \\
\notag &\qquad - \frac{1}{2} e^{-2 F} \nabla_\beta f ( \nabla^\mu \phi - F' \nabla^\mu f \cdot \phi ) ( \nabla_\mu \phi - F' \nabla_\mu f \cdot \phi ) \\
\notag &= e^{-2 F} \left( S \phi \nabla_\beta \phi - \frac{1}{2} \nabla_\beta f \cdot \nabla^\mu \phi \nabla_\mu \phi \right) - e^{-2 F} f F' \cdot \phi \nabla_\beta \phi \\
\notag &\qquad + \frac{1}{2} e^{-2 F} f \nabla_\beta f ( F' )^2 \cdot \phi^2 \text{.}
\end{align}
Using \eqref{eq.deriv_psi_phi}, we also see that
\begin{align}
\label{eql.algebraic_42} \frac{n - 1}{4} \cdot \psi \nabla_\beta \psi &= \frac{n - 1}{4} e^{-2 F} ( \phi \nabla_\beta \phi - F' \nabla_\beta f \cdot \phi^2 ) \text{.}
\end{align}
Since the definition of $\mc{A}$ yields
\begin{align}
\label{eql.algebraic_43} \frac{1}{2} \mc{A} \nabla_\beta f \cdot \psi^2 = \frac{1}{2} e^{-2 F} [ f (F')^2 - G_F ] \nabla_\beta f \cdot \phi^2 \text{.}
\end{align}
then combining \eqref{eql.algebraic_41}-\eqref{eql.algebraic_43} yields \eqref{eql.algebraic_4} and completes the proof.
\end{proof}

\begin{lemma} \label{thm.ptwise}
The following pointwise inequality holds,
\begin{align}
\label{eq.ptwise} \frac{1}{8} | F' |^{-1} | \mc{L}_U \psi |^2 &\geq ( f | F' | G_F - H_F ) \cdot \psi^2 - \mc{B}_U^F - \nabla^\beta P^F_\beta \text{.}
\end{align}
\end{lemma}

\begin{proof}
From \eqref{eq.algebraic}, we have
\begin{align}
\label{eql.ptwise_1} - \mc{L}_U \psi S_\ast \psi &= 2 | F^\prime | | S_\ast \psi |^2 + ( f | F' | G_F - H_F ) \psi^2 - \mc{B}_U^F - \nabla^\beta P^F_\beta \text{,}
\end{align}
The inequality \eqref{eq.ptwise} follows immediately from \eqref{eql.ptwise_1} and the basic inequality
\begin{align*}
- \mc{L}_U \psi S_\ast \psi &\leq \frac{1}{8} | F' |^{-1} | \mc{L} \psi |^2 + 2 | F' | | S_\ast \psi |^2 \text{.} \qedhere
\end{align*}
\end{proof}

To complete the proof of Proposition \ref{thm.carleman_gen}, we integrate \eqref{eq.ptwise} over $\Omega$ and apply the divergence theorem, \eqref{eq.divg_thm}, to the last term on the right-hand side of \eqref{eq.ptwise}.

\subsection{The Linear Estimate} \label{sec:carleman_linear}

We now derive Carleman-type estimates for the wave operator $\Box$ (with no potential).
In terms of the terminology presented in Proposition \ref{thm.carleman_gen}, we wish to consider the situation in which $U \equiv 0$, so that we have no positive bulk contribution from $U$, i.e., $\mc{B}_U^F \equiv 0$.

Ideally, the reparametrization of $f$ we would like to take is $F = - a \log f$ (corresponding to power law decay for the wave at infinity), where $a > 0$.
However, for this $F$, we see that the quantities $G_F$ and $H_F$, defined in \eqref{eq.G}, vanish identically, so that \eqref{eq.carleman_gen} produces no positive bulk terms at all.
Thus, we must add correction terms to the above $F$ in order to generate the desired positive bulk.

Furthermore, to ensure that these corrections remain everywhere lower order, we must construct separate reparametrizations for regions with $f$ small ($f < 1$) and with $f$ large ($f > 1$).
We must also ensure that these two reparametrizations match at the boundary $f = 1$.
These considerations motivate the definitions below:

\begin{definition} \label{def.abp_lh_pre}
Fixing constants $a, b, p \in \R$ satisfying
\begin{align}
\label{eq.abp_lh} a > 0 \text{,} \qquad 0 < p < 2 a \text{,} \qquad 0 \leq b < \frac{1}{4} \min ( 2 a - p, 4p ) \text{.}
\end{align}
we define the reparametrizations
\begin{align}
\label{eq.F_lh} F_\pm := -(a \pm b) \log f - \frac{b}{p} f^{\mp p} \text{.}
\end{align}
\end{definition}

In particular, $F_-$ will be our desired reparametrization on $\{f < 1\}$, while $F_+$ 
will be applicable in the opposite region $\{f > 1\}$.
By applying Proposition \ref{thm.carleman_gen} with $F_\pm$ and $U \equiv 0$, 
we will derive the following inequalities:

\begin{theorem} \label{thm.carleman_lh}
Let $\phi \in \mc{C}^2 (\mc{D})$, and fix $a, b, p \in \R$ satisfying \eqref{eq.abp_lh}.
Let $\Omega \subseteq \mc{D}$ be an admissible region, and partition $\Omega$ as
\begin{align*}
\Omega_l := \{ Q \in \Omega \mid f (Q) < 1 \} \text{,} \qquad \Omega_h := \{ Q \in \Omega \mid f (Q) > 1 \} \text{.}
\end{align*}
Then, there exist constants $C, K > 0$ such that:
\begin{align}
\label{eq.carleman_l} &C b^2 p \int_{\Omega_l} f^{ 2 (a - b) } f^{ p - 1 } \phi^2 \leq K a^{-1} \int_{\Omega_l} f^{ 2 (a - b) } \cdot f | \Box \phi |^2 + \int_{ \partial \Omega_l } P^-_\beta \mc{N}^\beta \text{,} \\
\label{eq.carleman_h} &C b^2 p \int_{\Omega_h} f^{ 2 (a + b) } f^{ -p - 1 } \phi^2 \leq K a^{-1} \int_{\Omega_h} f^{ 2 (a + b) } \cdot f | \Box \phi |^2 + \int_{ \partial \Omega_h } P^+_\beta \mc{N}^\beta \text{,}
\end{align}
where $\mc{N}$ denotes the oriented unit normals of $\partial \Omega_l$ and $\partial \Omega_h$, and where the $1$-forms $P^-$ and $P^+$ are defined via the following formula:
\begin{align}
\label{eq.carleman_lh_current} P^\pm_\beta &:= e^{-2 F_\pm} \left( \nabla^\alpha f \cdot \nabla_\alpha \phi \nabla_\beta \phi - \frac{1}{2} \nabla_\beta f \cdot \nabla^\mu \phi \nabla_\mu \phi \right) \\
\notag &\qquad + e^{-2 F_\pm} \left( \frac{n - 1}{4} - f F_\pm' \right) \cdot \phi \nabla_\beta \phi \\
\notag &\qquad + e^{-2 F_\pm} \left[ \left( f F_\pm' - \frac{n - 1}{4} \right) F_\pm' - \frac{1}{2} b p f^{\mp p - 1} \right] \nabla_\beta f \cdot \phi^2 \text{.}
\end{align}
Furthermore, on the middle boundary $\{f = 1\}$, we have that
\begin{align}
\label{eq.carleman_lh_match} P^- |_{ f = 1 } = P^+ |_{ f = 1 } \text{.}
\end{align}
\end{theorem}

\subsubsection{Proof of Theorem \ref{thm.carleman_lh}}

We begin with some elementary computations.

\begin{lemma} \label{thm.abp_lh_comp}
The following inequalities hold:
\begin{align}
\label{eq.abp_lh_comp} b < \frac{a}{2} \text{,} \qquad a \pm b \simeq a \text{,} \qquad a - b - \frac{1}{2} p > b \text{.}
\end{align}
\end{lemma}

\begin{proof}
The first inequality follows from \eqref{eq.abp_lh}, and the comparison $a \pm b \simeq a$ follows immediately from this.
For the remaining inequality, we apply \eqref{eq.abp_lh} twice:
\begin{align*}
a - b - \frac{p}{2} &> a - \frac{1}{4} ( 2 a - p ) - \frac{p}{2} = \frac{1}{2} a - \frac{1}{4} p > b \text{.} \qedhere
\end{align*}
\end{proof}

\begin{lemma} \label{thm.F_lh_deriv}
The following identities hold for $F_\pm$:
\begin{align}
\label{eq.F_lh_deriv} F_\pm' = - (a \pm b) f^{-1} \pm b f^{\mp p - 1} \text{,}
\end{align}
Furthermore, recalling the notations in \eqref{eq.G}, we have that
\begin{align}
\label{eq.GH_lh} G_{ F_\pm } = b p f^{\mp p - 1} \text{,} \qquad H_{ F_\pm } = \mp \frac{1}{2} b p^2 f^{\mp p - 1} \text{.}
\end{align}
In particular, on the level set $\mc{F}_1 = \{ f = 1 \}$, we have
\begin{align}
\label{eq.F_lh_match} F_+ |_{ \mc{F}_1 } = F_- |_{ \mc{F}_1 } \text{,} \qquad F_+' |_{ \mc{F}_1 } = F_-' |_{ \mc{F}_1 } \text{,} \qquad G_{F_+} |_{ \mc{F}_1 } = G_{F_-} |_{ \mc{F}_1 } \text{.}
\end{align}
\end{lemma}

\begin{proof}
These are direct computations.
\end{proof}

\begin{lemma} \label{thm.F_lh_comp}
The following comparisons hold:
\begin{itemize}
\item If $0 < f \leq 1$, then
\begin{align}
\label{eq.F_l_comp} f^{a - b} < e^{-F_-} \leq e f^{a - b} \text{,} \qquad -a f^{-1} \leq F_-' < - (a - b) f^{-1} \text{.}
\end{align}

\item If $1 \leq f < \infty$, then
\begin{align}
\label{eq.F_h_comp} f^{a + b} < e^{-F_+} \leq e f^{a + b} \text{,} \qquad - (a + b) f^{-1} < F_+' \leq -a f^{-1} \text{.}
\end{align}
\end{itemize}
In particular, \eqref{eq.F_l_comp} implies that $F_-$ is inward-directed whenever $f < 1$, while \eqref{eq.F_h_comp} implies $F_+$ is inward-directed whenever $f > 1$.
\end{lemma}

\begin{proof}
The comparisons \eqref{eq.F_l_comp} and \eqref{eq.F_h_comp} follow immediately from \eqref{eq.F_lh}, \eqref{eq.F_lh_deriv}, and the trivial inequality $b p^{-1} < 1$, which is a consequence of \eqref{eq.abp_lh}.
The remaining monotonicity properties follow from \eqref{eq.abp_lh_comp}, \eqref{eq.F_l_comp}, and \eqref{eq.F_h_comp}.
\end{proof}

\begin{lemma} \label{thm.bulk_lh}
The following inequalities hold:
\begin{itemize}
\item If $0 < f < 1$, then
\begin{align}
\label{eq.bulk_l} f | F_-' | G_{ F_- } - H_{ F_- } > b^2 p f^{p - 1} > 0 \text{.}
\end{align}

\item If $1 < f < \infty$, then
\begin{align}
\label{eq.bulk_h} f | F_+' | G_{ F_+ } - H_{ F_+ } > b^2 p f^{-p - 1} > 0 \text{.}
\end{align}
\end{itemize}
\end{lemma}

\begin{proof}
First, in the case $0 < f < 1$, we have
\begin{align}
\label{eql.bulk_1} f | F_-' | G_{ F_- } - H_{ F_- } &= ( a - b + b f^p ) bp f^{p-1} - \frac{1}{2} b p^2 f^{p-1} \\
\notag &= b p f^{p-1} \left( a - b - \frac{1}{2} p + b f^p \right) \\
\notag &\geq b p \left( a - b - \frac{1}{2} p \right) f^{p-1} \text{.}
\end{align}
Similarly, when $1 < f < \infty$, we have
\begin{align}
\label{eql.bulk_2} f | F_+' | G_{ F_+ } - H_{ F_+ } &= ( a + b - b f^{-p} ) bp f^{-p-1} + \frac{1}{2} b p^2 f^{-p-1} \\
\notag &\geq b p \left( a + \frac{1}{2} p \right) f^{-p - 1} \text{.}
\end{align}
The desired inequalities now follow by applying \eqref{eq.abp_lh_comp} to \eqref{eql.bulk_1} and \eqref{eql.bulk_2}.
\end{proof}

We now complete the proof of Theorem \ref{thm.carleman_lh}.
First, for \eqref{eq.carleman_l}, we apply Theorem \ref{thm.carleman_gen} with $F = F_-$ and $U \equiv 0$.
Combining this with \eqref{eq.abp_lh_comp}, \eqref{eq.F_l_comp}, and \eqref{eq.bulk_l} yields
\begin{align}
\label{eql.carleman_l_2} &C b^2 p \int_{ \Omega_l } f^{ 2 (a - b) } f^{p - 1} \phi^2 \leq K a^{-1} \int_{ \Omega_l } f^{2 (a - b) } f | \Box \phi |^2 + \int_{ \partial \Omega_l } P^{F_-}_\beta \mc{N}^\beta \text{,}
\end{align}
where $C$ and $K$ are constants, and where $\mc{N}$, $P^{F_-}$ are as defined in Theorem \ref{thm.carleman_gen}.
Since $U \equiv 0$, then $P^{F_-}$ is precisely the one-form $P^-$ in \eqref{eq.carleman_lh_current}, proving \eqref{eq.carleman_l}.

Similarly, for \eqref{eq.carleman_h}, we apply Theorem \ref{thm.carleman_gen} with $F = F_+$ and $U \equiv 0$, and we combine the result with \eqref{eq.abp_lh_comp}, \eqref{eq.F_h_comp}, and \eqref{eq.bulk_h}, which yields
\begin{align}
\label{eql.carleman_h_2} &C b^2 p \int_{ \Omega_h } f^{ 2 (a + b) } f^{-p - 1} \phi^2 \leq K a^{-1} \int_{ \Omega_h } f^{2 (a + b) } f | \Box \phi |^2 + \int_{ \partial \Omega_h } P^{F_+}_\beta \mc{N}^\beta \text{,}
\end{align}
Since $P^{F_+}$ is precisely $P^+$, we obtain \eqref{eq.carleman_h}.

Finally, \eqref{eq.carleman_lh_match} is an immediate consequence of \eqref{eq.carleman_lh_current} and \eqref{eq.F_lh_match}

\subsection{The Nonlinear Estimate} \label{sec:carleman_nonlinear}

We next discuss improved estimates for nonlinear wave equations, in particular those found in Theorems \ref{thm.uc_focusing} and \ref{thm.uc_defocusing}.
With respect to the terminology within Proposition \ref{thm.carleman_gen}, we consider $U \in \mc{C}^1 (\mc{D} \times \R)$ of the form
\begin{align}
\label{eq.potential_nonlinear} U (Q, \phi) = \pm \frac{1}{p+1} V (Q) \cdot | \phi |^{p+1} \text{,} \qquad p \geq 1 \text{,}
\end{align}
where $V \in \mc{C}^1 (\mc{D})$ is strictly positive.
From Definition \ref{def.wave_op}, this corresponds to
\begin{align}
\label{eq.wave_nonlinear} \Box_U \phi = \Box \phi \pm V \cdot | \phi |^{p - 1} \phi \text{,} \qquad p \geq 1 \text{.}
\end{align}

Since we will be expecting positive bulk terms arising solely from $U$ (that is, $- \mc{B}^F_U > 0$ in \eqref{eq.carleman_gen}), we no longer require the correction terms used throughout Section \ref{sec:carleman_linear} for our reparametrizations of $f$.
In other words, we can simply use
\begin{align}
\label{eq.F_0} F_0 = -a \log f \text{,} \qquad a > 0 \text{.}
\end{align}
In particular, we need not consider the regions $\{ f > 1 \}$ and $\{ f < 1 \}$ separately.
This makes some aspects of the analysis much simpler compared to Theorem \ref{thm.carleman_lh}.

The Carleman-type estimate we will derive is the following:

\begin{theorem} \label{thm.carleman_nl}
Let $\phi \in \mc{C}^2 (\mc{D})$,
 and let $\Omega \subseteq \mc{D}$ be an admissible region.
Furthermore, let $p \geq 1$, and let $V \in \mc{C}^1 (\mc{D})$ be strictly positive.
Then,
\begin{align}
\label{eq.carleman_nl} \pm \frac{1}{p + 1} \int_\Omega f^{2a} \cdot V \Gamma_V \cdot | \phi |^{p + 1} &\leq \frac{1}{8 a} \int_\Omega f^{2 a} f \cdot | \Box_V^\pm \phi |^2 + \int_{ \partial \Omega } P^{\pm V}_\beta \mc{N}^\beta \text{,}
\end{align}
where $\mc{N}$ is the oriented unit normal to $\partial \Omega$, and where:
\begin{align}
\label{eq.bulk_gamma} \Box_V^\pm \phi &:= \Box \phi \pm V | \phi |^{p - 1} \phi \text{,} \\
\notag \Gamma_V &:= \nabla^\alpha f \nabla_\alpha (\log V) - \frac{n - 1 + 4 a}{4} \left( p - 1 - \frac{4}{n - 1 + 4a} \right) \text{,} \\
\notag P^{\pm V}_\beta &:= f^{2 a} \left( \nabla^\alpha f \cdot \nabla_\alpha \phi \nabla_\beta \phi - \frac{1}{2} \nabla_\beta f \cdot \nabla^\mu \phi \nabla_\mu \phi \right) \\
\notag &\qquad \pm \frac{1}{p + 1} f^{2 a} \nabla_\beta f \cdot V | \phi |^{p + 1} + \left( \frac{n - 1}{4} + a \right) f^{2 a} \cdot \phi \nabla_\beta \phi \\
\notag &\qquad + a \left( \frac{n - 1}{4} + a \right) f^{2 a} f^{-1} \nabla_\beta f \cdot \phi^2 \text{.}
\end{align}
\end{theorem}

The remainder of this section is dedicated to the proof of Theorem \ref{thm.carleman_nl}.

\subsubsection{Proof of Theorem \ref{thm.carleman_nl}}

The main new task is to examine the bulk term $\mc{B}_U^{F_0}$ (see \eqref{eq.bulk}) arising from the $U$ defined in \eqref{eq.potential_nonlinear}.
From a direct computation using \eqref{eq.bulk} and \eqref{eq.potential_nonlinear}, we obtain the following:

\begin{lemma} \label{thm.nonlinear_bulk}
Let $U$ and $F_0$ be as in \eqref{eq.potential_nonlinear} and \eqref{eq.F_0}.
Then,
\begin{align}
\label{eq.nonlinear_bulk} - \mc{B}_U^{ F_0 } &= \pm \frac{1}{p + 1} f^{2a} V \cdot \Gamma_V \cdot | \phi |^{p+1} \text{,}
\end{align}
where $\mc{B}^{F_0}_U$ is as defined in \eqref{eq.bulk}, and $\Gamma_V$ is as in \eqref{eq.bulk_gamma}.
\end{lemma}

\begin{proof}
For an arbitrary reparametrization, we compute, using \eqref{eq.bulk} and \eqref{eq.potential_nonlinear},
\begin{align}
\mc{B}_U^F &= \pm e^{-2 F} \left( \frac{n - 1}{4} - f F' \right) V \cdot | \phi |^{p+1} \mp \frac{1}{p+1} e^{-2 F} \nabla^\alpha f \nabla_\alpha V \cdot | \phi |^{p+1} \\
\notag &\qquad \mp \frac{2}{p + 1} e^{-2 F} \left( \frac{n + 1}{4} - f F' \right) V \cdot | \phi |^{p+1} \\
\notag &= \mp \frac{1}{p + 1} e^{-2 F} [ \nabla^\alpha f \nabla_\alpha V - p^\ast V + (p - 1) f F' V ] \cdot | \phi |^{p+1} \text{,}
\end{align}
where
\begin{align*}
p^\ast = \frac{ (p + 1) (n - 1) }{4} - \frac{n + 1}{2} \text{.}
\end{align*}
Substituting $F_0$ for $F$, and noting that
\begin{align*}
f F_0' \equiv -a \text{,} \qquad e^{-2 F} = f^{2 a} \text{,}
\end{align*}
we immediately obtain \eqref{eq.nonlinear_bulk}.
\end{proof}

As a result, $-\mc{B}_U^{ F_0 }$ is strictly positive $\mc{D}$ if and only if $\pm \Gamma_V > 0$.

\begin{remark}
The computations in Lemma \ref{thm.nonlinear_bulk} can readily be generalized.
For example, one can consider wave operators of the form
\begin{align}
\label{eq.wave_nonlinear_gen} U (Q, \phi) = \pm V (Q) W (\phi) \text{,} \qquad \Box_U \phi = \Box \phi \pm V \cdot \dot{W} (\phi) \text{,}
\end{align}
where we also assume $V (Q) \cdot W (\phi) > 0$ for all $(Q, \phi)$.
(These correspond to further generalizations of focusing and defocusing wave operators.)
From analogous calculations, one sees that $-\mc{B}_U^{ F_0 }$ is everywhere positive if
\begin{itemize}
\item There is some $p \geq 1$ such that $\pm \Gamma_V > 0$.

\item $W (\phi)$ grows at most as quickly as $| \phi |^{p + 1}$ when $V W$ is positive.

\item $W (\phi)$ grows at least as quickly as $| \phi |^{p + 1}$ when $V W$ is negative.
\end{itemize}
Such statements can be even further extended to more general $U$, but precise formulations of these statements tend to be more complicated.
\end{remark}

Theorem \ref{thm.carleman_nl} now follows by applying \eqref{eq.carleman_gen}---with $F = F_0 = - a \log f $ and $U$ as in \eqref{eq.potential_nonlinear}---and then by expanding $\mc{B}_U^{F_0}$ using \eqref{eq.nonlinear_bulk}.

\section{Proofs of the Main Results} \label{sec:uc}

The goal of this section is to prove the global uniqueness results---Theorems \ref{thm.uc_finite}, \ref{thm.uc_focusing}, and \ref{thm.uc_defocusing}---from Section \ref{sec:intro_results}.
The main steps will be to apply the Carleman-type estimates from the preceding section: Theorems \ref{thm.carleman_lh} for the proof of Theorem \ref{thm.uc_finite}, and Theorem \ref{thm.carleman_nl} for the proofs of Theorems \ref{thm.uc_focusing} and \ref{thm.uc_defocusing}.

Note first of all that the weight $(1 + |u|) (1 + |v|)$ can be written as
\begin{align}
\label{eq.radiation_field} (1 + |u|) (1 + |v|) = (1 + r + f) \text{.}
\end{align}
Thus, the decay conditions \eqref{eq.wave_decay} can be more conveniently expressed as
\begin{align}
\label{eq.wave_decay_ex} \sup_{ \mc{D} } \left[ ( 1 + r + f )^\frac{n - 1 + \delta}{2} ( | u \cdot \partial_u \phi | + | v \cdot \partial_v \phi | ) \right] < \infty \text{,} \\
\notag \sup_{ \mc{D} \cap \{ f < 1 \} } \left[ ( 1 + r )^\frac{n - 1 + \delta}{2} f^\frac{1}{2} | \nasla \phi | \right] < \infty \text{,} \\
\notag \sup_{ \mc{D} } \left[ ( 1 + r + f )^\frac{n - 1 + \delta}{2} | \phi | \right] < \infty \text{,}
\end{align}
while the special decay condition \eqref{eq.wave_decay_focusing} is equivalent to
\begin{align}
\label{eq.wave_decay_focusing_ex} \sup_{ \mc{D} \cap \{ f > 1 \} } \left[ (1 + r + f)^\frac{n - 1 + \delta}{p + 1} f^\frac{1}{p + 1} V^\frac{1}{p + 1} | \phi | \right] < \infty \text{.}
\end{align}
From now on, we will refer to \eqref{eq.wave_decay_ex} and \eqref{eq.wave_decay_focusing_ex} as our decay assumptions.

\subsection{Special Domains}

The first preliminary step is to define the admissible regions on which we apply our estimates.
A natural choice for this $\Omega$ would be domains with level sets of $f$ as its boundary.
We denote these level sets by
\begin{align}
\label{eq.levelf} \mc{F}_\omega := \{ Q\in \mc{D} \mid f (Q) = \omega \} \text{.}
\end{align}
Observe that the $\mc{F}_\omega$'s, for all $0 < \omega < \infty$, form a family of timelike hyperboloids terminating at the corners of $\mc{D}$ on future and past null infinity.
\footnote{See Figure \ref{fig.penrose}.}

The $\mc{F}_\omega$'s are useful here since they characterize the boundary of $\mc{D}$ in the limit.
Indeed, in the Penrose-compactified sense, $\mc{F}_\omega$ tends toward the null cone about the origin as $\omega \searrow 0$, and $\mc{F}_\omega$ tends toward the outer half of null infinity as $\omega \nearrow \infty$.

However, one defect in the above is that the region between two $\mc{F}_\omega$'s fails to be bounded.
As a result, we define an additional function
\begin{align}
\label{eq.h} h \in \mc{C}^\infty (\mc{D}) \text{,} \qquad h := - \frac{v}{u} = \frac{ r + t }{ r - t } \text{,}
\end{align}
whose level sets we denote by
\begin{align}
\label{eq.levelh} \mc{H}^\tau := \{ Q \in \mc{D} \mid h (Q) = \tau \} \text{.}
\end{align}
The $\mc{H}^\tau$'s, for $0 < \tau < \infty$, form a family of spacelike cones terminating at the origin and at spacelike infinity.
Moreover, in the Penrose-compactified picture:
\begin{itemize}
\item As $\tau \nearrow \infty$, the $\mc{H}^\tau$'s tend toward both the future null cone about the origin and the outer half of future null infinity.

\item As $\tau \searrow 0$, the $\mc{H}^\tau$'s tend toward both the past null cone about the origin and the outer half of past null infinity.
\end{itemize}

The regions we wish to consider are those bounded by level sets of $f$ and $h$.
More specifically, given $0 < \rho < \omega < \infty$ and $0 < \sigma < \tau < \infty$, we define
\begin{align}
\label{eq.region_exhaustion} \mc{D}^{\sigma, \tau}_{\rho, \omega} := \{ Q \in \mc{D} \mid \rho < f (Q) < \omega \text{, } \sigma < h (Q) < \tau \} \text{,}
\end{align}
We also define corresponding cutoffs to the $\mc{F}_\omega$'s and $\mc{H}^\tau$'s:
\begin{align}
\label{eq.level_exhaustion} \mc{F}^{\sigma, \tau}_\omega := \{ Q \in \mc{F}_\omega \mid \sigma < h (Q) < \tau \} \text{,} \qquad \mc{H}^\tau_{\rho, \omega} := \{ Q \in \mc{H}^\tau \mid \rho < f (Q) < \omega \} \text{.}
\end{align}

\subsubsection{Basic Properties}

We begin by listing some properties of $f$ and $h$ that will be needed in upcoming computations.
First, the derivative of $h$ satisfy the following:

\begin{lemma} \label{thm.h_deriv}
The following identities hold:
\begin{align}
\label{eq.h_deriv} \partial_u h = \frac{v}{u^2} \text{,} \qquad \partial_v h = - \frac{1}{u} \text{.}
\end{align}
As a result,
\begin{align}
\label{eq.h_grad} \grad h = \frac{1}{2} u^{-2} ( u \cdot \partial_u - v \cdot \partial_v ) \text{,} \qquad \nabla^\alpha h \nabla_\alpha h = - u^{-4} f \text{,} \qquad \nabla^\alpha h \nabla_\alpha f = 0 \text{.}
\end{align}
\end{lemma}

In particular, observe that \eqref{eq.f_grad} implies the $\mc{F}_\omega$'s are timelike, while \eqref{eq.h_grad} implies the $\mc{H}^\tau$'s are spacelike.
Furthermore, the last identity in \eqref{eq.h_grad} implies that the $\mc{F}_\omega$'s and $\mc{H}^\tau$'s are everywhere orthogonal to each other.

Next, observe the region $\mc{D}^{\sigma, \tau}_{\rho, \omega}$ has piecewise smooth boundary, with
\begin{align}
\label{eq.exhaustion_boundary} \partial \mc{D}^{\sigma, \tau}_{\rho, \omega} = \mc{F}^{\sigma, \tau}_\omega \cup \mc{F}^{\sigma, \tau}_\rho \cup \mc{H}^\tau_{\rho, \omega} \cup \mc{H}^\sigma_{\rho, \omega} \text{,}
\end{align}
hence it is indeed an admissible region.
Also, from \eqref{eq.f_grad} and \eqref{eq.h_grad}, we see that:
\begin{itemize}
\item On $\mc{F}^{\sigma, \tau}_\omega$ and $\mc{F}^{\sigma, \tau}_\rho$, the outer unit normals with respect to $\mc{D}^{\sigma, \tau}_{\rho, \omega}$ are
\begin{align}
\label{eq.exhaustion_normal_f} \mc{N} ( \mc{F}^{\sigma, \tau}_\omega ) = N := f^{-\frac{1}{2}} \grad f \text{,} \qquad \mc{N} ( \mc{F}^{\sigma, \tau}_\rho ) = - N = - f^{-\frac{1}{2}} \grad f \text{.}
\end{align}

\item On $\mc{H}^\tau_{\rho, \omega}$ and $\mc{H}^\sigma_{\rho, \omega}$, the inner unit normals with respect to $\mc{D}^{\sigma, \tau}_{\rho, \omega}$ are
\begin{align}
\label{eq.exhaustion_normal_h} \mc{N} ( \mc{H}^\tau_{\rho, \omega} ) = -T := u^2 f^{-\frac{1}{2}} \grad h \text{,} \qquad \mc{N} ( \mc{H}^\sigma_{\rho, \omega} ) = T = - u^2 f^{-\frac{1}{2}} \grad h \text{.}
\end{align}
\end{itemize}
In view of the above, we obtain the following:

\begin{lemma} \label{thm.divg_fh}
If $P$ is a continuous $1$-form on $\mc{D}$, and if $\mc{N}$ is the oriented unit normal for $\partial \mc{D}^{\sigma, \tau}_{\rho, \omega}$, then the following identity holds:
\begin{align}
\label{eq.divg_fh} \int_{ \partial \mc{D}^{\sigma, \tau}_{\rho, \omega} } P_\beta \mc{N}^\beta &= \int_{ \mc{F}^{\sigma, \tau}_\omega } f^{-\frac{1}{2}} P_\beta \nabla^\beta f - \int_{ \mc{F}^{\sigma, \tau}_\rho } f^{-\frac{1}{2}} P_\beta \nabla^\beta f \\
\notag &\qquad + \int_{ \mc{H}^\tau_{\rho, \omega} } u^2 f^{-\frac{1}{2}} P_\beta \nabla^\beta h - \int_{ \mc{H}^\sigma_{\rho, \omega} } u^2 f^{-\frac{1}{2}} P_\beta \nabla^\beta h \text{.}
\end{align}
\end{lemma}

Finally, we note that the level sets of $(f, h)$ are simply the level spheres of $(t, r)$, and the values of these functions can be related as follows:

\begin{lemma} \label{thm.spheres_fh}
Given $Q \in \mc{D}$, we have that $( f (Q), h (Q) ) = (\omega, \tau)$ if and only if
\begin{align}
\label{eq.spheres_fh} v (Q) = \omega^\frac{1}{2} \tau^\frac{1}{2} \text{,} &\qquad u (Q) = - \omega^\frac{1}{2} \tau^{-\frac{1}{2}} \text{,} \\
\notag r (Q) = \omega^\frac{1}{2} ( \tau^\frac{1}{2} + \tau^{-\frac{1}{2}} ) \text{,} &\qquad t (Q) = \omega^\frac{1}{2} ( \tau^\frac{1}{2} - \tau^{-\frac{1}{2}} ) \text{.}
\end{align}
\end{lemma}

\subsubsection{Boundary Expansions}

In light of Lemma \ref{thm.divg_fh} and the Carleman-type estimates from Section \ref{sec:carleman}, we will need to bound integrands of the form $P_\beta \nabla^\beta f$ and $P_\beta \nabla^\beta h$, where $P$ is one of the currents $P^\pm$ (see \eqref{eq.carleman_lh_current}) or $P^{\pm V}$ (see \eqref{eq.bulk_gamma}).

\begin{lemma} \label{thm.boundary_lh_est}
Let $P^\pm$ be as in \eqref{eq.carleman_lh_current}.
Then, there exists $K > 0$ such that:
\begin{itemize}
\item In the region $\{ f < 1 \}$,
\begin{align}
\label{eq.boundary_l_est} - P^-_\beta \nabla^\beta f &\leq K f^{ 2 (a - b) } [ f \cdot | \nasla \phi |^2 + (n + a)^2 \cdot \phi^2 ] \text{,} \\
\notag | u^2 P^-_\beta \nabla^\beta h | &\leq K f^{ 2 (a - b) } [ ( u \cdot \partial_u \phi )^2 + ( v \cdot \partial_v \phi )^2 + (n + a)^2 \cdot \phi^2 ] \text{.}
\end{align}

\item In the region $\{ f > 1 \}$,
\begin{align}
\label{eq.boundary_h_est} P^+_\beta \nabla^\beta f &\leq K f^{ 2 (a + b) } [ ( u \cdot \partial_u \phi )^2 + ( v \cdot \partial_v \phi )^2 + (n + a)^2 \cdot \phi^2 ] \text{,} \\
\notag | u^2 P^+_\beta \nabla^\beta h | &\leq K f^{ 2 (a + b) } [ ( u \cdot \partial_u \phi )^2 + ( v \cdot \partial_v \phi )^2 + (n + a)^2 \cdot \phi^2 ] \text{.}
\end{align}
\end{itemize}
\end{lemma}

\begin{proof}
Applying \eqref{eq.f_deriv}, \eqref{eq.f_grad}, \eqref{eq.h_deriv}, and \eqref{eq.h_grad} to the definition \eqref{eq.carleman_lh_current} (and noting in particular that $\grad f$ and $\grad h$ are everywhere orthogonal), we expand
\begin{align}
\label{eql.boundary_lh_est_pre} P^\pm_\beta \nabla^\beta f &= \frac{1}{4} e^{-2 F_\pm} [ ( u \cdot \partial_u \phi )^2 + ( v \cdot \partial_v \phi )^2 ] - \frac{1}{2} e^{-2 F_\pm} f \cdot | \nasla \phi |^2 \\
\notag &\qquad + \frac{1}{2} e^{-2 F_\pm} \left( \frac{n - 1}{4} - f F_\pm' \right) \cdot \phi ( u \cdot \partial_u \phi + v \cdot \partial_v \phi ) \\
\notag &\qquad - e^{-2 F_\pm} \left[ \left( \frac{n - 1}{4} - f F_\pm' \right) f F_\pm' - \frac{1}{2} b p f^{\mp p} \right] \cdot \phi^2 \text{,} \\
\notag u^2 P^\pm_\beta \nabla^\beta h &= \frac{1}{4} e^{-2 F_\pm} [ ( u \cdot \partial_u \phi )^2 - ( v \cdot \partial_v \phi )^2 ] \\
\notag &\qquad + \frac{1}{2} e^{-2 F_\pm} \left( \frac{n - 1}{4} - f F_\pm' \right) \cdot \phi ( u \cdot \partial_u \phi - v \cdot \partial_v \phi ) \text{.}
\end{align}

Next, we note the inequality
\begin{align}
\label{eql.boundary_lh_est_0} &\left| \frac{1}{2} e^{-2 F_\pm} \left( \frac{n - 1}{4} - f F_\pm^\prime \right) \cdot \phi ( u \cdot \partial_u \phi \pm v \cdot \partial_v \phi ) \right| \\
\notag &\quad \leq \frac{1}{4} e^{-2 F_\pm} [ ( u \cdot \partial_u \phi )^2 + ( v \cdot \partial_v \phi )^2 ] + \frac{1}{2} e^{-2 F_\pm} \left( \frac{n - 1}{4} - f F_\pm^\prime \right)^2 \cdot \phi^2 \text{.}
\end{align}
Applying \eqref{eql.boundary_lh_est_0} to each of the identities in \eqref{eql.boundary_lh_est_pre} and then dropping any purely nonpositive terms on the right-hand side, we obtain
\begin{align}
\label{eql.boundary_lh_est_1} P^\pm_\beta \nabla^\beta f &\leq K e^{-2 F_\pm} [ ( u \cdot \partial_u \phi )^2 + ( v \cdot \partial_v \phi )^2 ] \\
\notag &\qquad + K e^{-2 F_\pm} [ n^2 + ( f F_\pm' )^2 + b p f^{\mp p} ] \cdot \phi^2 \text{,} \\
\notag - P^\pm_\beta \nabla^\beta f &\leq K e^{-2 F_\pm} f \cdot | \nasla \phi |^2 + K e^{-2 F_\pm} [ n^2 + ( f F_\pm' )^2 + b p f^{\mp p} ] \cdot \phi^2 \text{,} \\
\notag | u^2 P^\pm_\beta \nabla^\beta h | &\leq K e^{-2 F_\pm} [ ( u \cdot \partial_u \phi )^2 + ( v \cdot \partial_v \phi )^2 ] + K e^{-2 F_\pm} [ n^2 + ( f F_\pm' )^2 ] \cdot \phi^2 \text{.}
\end{align}

Now, recall from Propositions \ref{thm.abp_lh_comp}-\ref{thm.F_lh_comp} that
\begin{align}
\label{eql.boundary_lh_est_2} \begin{cases} f^2 ( F_-' )^2 + b p f^p \lesssim a^2 \text{,} \quad e^{-2 F_-} \lesssim f^{ 2 (a - b) } & \qquad f < 1 \text{,} \\ f^2 ( F_+' )^2 + b p f^{-p} \lesssim a^2 \text{,} \quad e^{-2 F_+} \lesssim f^{ 2 (a + b) } & \qquad f > 1 \text{.} \end{cases}
\end{align}
Combining \eqref{eql.boundary_lh_est_1} and \eqref{eql.boundary_lh_est_2} results in both \eqref{eq.boundary_l_est} and \eqref{eq.boundary_h_est}.
\end{proof}

\begin{lemma} \label{thm.boundary_nl_est}
Let $P^{\pm V}$ be as in \eqref{eq.bulk_gamma}.
Then, there exists $K > 0$ such that:
\begin{align}
\label{eq.boundary_nl_est} P^{\pm V}_\beta \nabla^\beta f &\leq K f^{ 2 a } [ ( u \cdot \partial_u \phi )^2 + ( v \cdot \partial_v \phi )^2 + (n + a)^2 \cdot \phi^2 ] \\
\notag &\qquad \pm (p + 1)^{-1} f^{2 a} f V \cdot | \phi |^{p+1} \text{,} \\
\notag - P^{\pm V}_\beta \nabla^\beta f &\leq K f^{ 2 a } [ f \cdot | \nasla \phi |^2 + (n + a)^2 \cdot \phi^2 ] \mp (p + 1)^{-1} f^{2 a} f V \cdot | \phi |^{p+1} \text{,} \\
\notag | u^2 P^{\pm V}_\beta \nabla^\beta h | &\leq K f^{ 2 a } [ ( u \cdot \partial_u \phi )^2 + ( v \cdot \partial_v \phi )^2 + (n + a)^2 \cdot \phi^2 ] \text{.}
\end{align}
\end{lemma}

\begin{proof}
The proof is analogous to that of Lemma \ref{thm.boundary_lh_est}.
Applying \eqref{eq.f_deriv}, \eqref{eq.f_grad}, \eqref{eq.h_deriv}, and \eqref{eq.h_grad} to \eqref{eq.bulk_gamma} results in the expansions
\begin{align}
\label{eql.boundary_nl_est_pre} P^{\pm V}_\beta \nabla^\beta f &= \frac{1}{4} f^{2 a} [ ( u \cdot \partial_u \phi )^2 + ( v \cdot \partial_v \phi )^2 ] - \frac{1}{2} f^{2 a} f \cdot | \nasla \phi |^2 \\
\notag &\qquad + \frac{1}{2} \left( \frac{n - 1}{4} + a \right) f^{2 a} \cdot \phi ( u \cdot \partial_u \phi + v \cdot \partial_v \phi ) \\
\notag &\qquad + a \left( \frac{n - 1}{4} + a \right) f^{2 a} \cdot \phi^2 \pm \frac{1}{p + 1} f^{2 a} f V \cdot | \phi |^{p+1} \text{,} \\
\notag u^2 P^{\pm V}_\beta \nabla^\beta h &= \frac{1}{4} f^{2 a} [ ( u \cdot \partial_u \phi )^2 - ( v \cdot \partial_v \phi )^2 ] \\
\notag &\qquad + \frac{1}{2} \left( \frac{n - 1}{4} + a \right) f^{2 a} \cdot \phi ( u \cdot \partial_u \phi - v \cdot \partial_v \phi ) \text{.}
\end{align}
Handling the cross-terms in \eqref{eql.boundary_nl_est_pre} using an analogue of \eqref{eql.boundary_lh_est_0} yields \eqref{eq.boundary_nl_est}.
\end{proof}

\subsection{Boundary Limits}

In order to convert our main estimates over the $\mc{D}^{\sigma, \tau}_{\rho, \omega}$'s into a unique continuation result, we will need to eliminate the resulting boundary terms.
To do this, we must take the limit of the boundary terms toward null infinity and the null cone about the origin, i.e., the boundary of $\mc{D}$.
More specifically, we wish to let $(\sigma, \tau) \rightarrow (0, \infty)$, and then $(\rho, \omega) \rightarrow (0, \infty)$.

\subsubsection{Coarea Formulas}

To obtain these necessary limits, we will need to express integrals over the $\mc{F}_\omega$'s and $\mc{H}^\tau$'s more explicitly.
For this, we derive coarea formulas below in order to rewrite these expressions in terms of spherical integrals.

In what follows, we will assume that integrals over $\Sph^{n-1}$ will always be with respect to the volume form associated with the (unit) round metric $\mathring{\gamma}$.

We can foliate $\mc{F}_\omega$ by level sets of $t$, which are $(n-1)$-spheres.
In other words,
\begin{align}
\label{eq.f_comp} \mc{F}_\omega \simeq \R \times \Sph^{n-1} \text{,}
\end{align}
where the $\R$-component corresponds to the $t$-coordinate, while the $\Sph^{n-1}$-component is the spherical value (as an element of a level set of $(t, r)$).
Furthermore, by \eqref{eq.spheres_fh}, restricting the correspondence \eqref{eq.f_comp} to finite cutoffs yields
\begin{align}
\label{eq.f_comp_cutoff} \mc{F}^{\sigma, \tau}_\omega \simeq ( \omega^\frac{1}{2} ( \sigma^\frac{1}{2} - \sigma^{-\frac{1}{2}} ), \omega^\frac{1}{2} ( \tau^\frac{1}{2} - \tau^{-\frac{1}{2}} ) ) \times \Sph^{n-1} \text{.}
\end{align}

We now wish to split integrals over $\mc{F}_\omega$ as in \eqref{eq.f_comp_cutoff}: as an integral first over $\Sph^{n-1}$ and then over $t$.
For convenience, we define, for any $\Psi \in \mc{C}^\infty ( \mc{D} )$, the shorthands
\begin{align}
\label{eq.f_int} \int_{ h = \sigma }^{ h = \tau } \int_{ \Sph^{n-1} } \Psi |_{ f = \omega } dt &:= \int_{ \omega^\frac{1}{2} ( \sigma^\frac{1}{2} - \sigma^{-\frac{1}{2}} ) }^{ \omega^\frac{1}{2} ( \tau^\frac{1}{2} - \tau^{-\frac{1}{2}} ) } \left[ \int_{ \Sph^{n-1} } \Psi |_{ (f, t) = (\omega, s) } \right] ds \text{,} \\
\notag \int_{-\infty}^\infty \int_{ \Sph^{n-1} } \Psi |_{ f = \omega } dt &:= \int_{-\infty}^\infty \left[ \int_{ \Sph^{n-1} } \Psi |_{ (f, t) = (\omega, s) } \right] ds \text{,}
\end{align}
with the second equation defined only when $\Psi$ is sufficiently integrable.
Note these are the integrals over $\mc{F}_\omega$ and $\mc{F}^{\sigma, \tau}_\omega$, respectively, in terms of level spheres of $t$.

\begin{proposition} \label{thm.coarea_f}
For any $\Psi \in \mc{C}^\infty ( \mc{D} )$, we have the identity
\begin{align}
\label{eq.coarea_f} \int_{ \mc{F}^{\sigma, \tau}_\omega } \Psi &= 2 \omega^\frac{1}{2} \int_{ h = \sigma }^{ h = \tau } \int_{ \Sph^{n-1} } \Psi r^{n-2} |_{ f = \omega } dt \text{.}
\end{align}
\end{proposition}

\begin{proof}
Let $D t$ denote the gradient of $t$ on $\mc{F}_\omega$, with respect to the metric induced by $g$.
By definition, $D t$ is tangent to $\mc{F}_\omega$ and normal to the level spheres of $(t, r)$.
The vector field $T$, defined in \eqref{eq.exhaustion_normal_h}, satisfies these same properties due to \eqref{eq.h_grad}, so $D t$ and $T$ point in the same direction.
Since $T$ is unit, it follows from \eqref{eq.exhaustion_normal_h} that
\begin{align}
\label{eql.grad_t_f} | D t |_g &= | g ( T, D t ) | = \frac{1}{2} \left( \sqrt{ \frac{-u}{v} } + \sqrt{ \frac{v}{-u} } \right) = \frac{1}{2} f^{-\frac{1}{2}} r \text{.}
\end{align}
Applying the coarea formula and \eqref{eq.spheres_fh} yields
\begin{align}
\label{eql.coarea_f_0} \int_{ \mc{F}^{\sigma, \tau}_\omega } \Psi &= \int_{ \omega^\frac{1}{2} ( \sigma^\frac{1}{2} - \sigma^{-\frac{1}{2}} ) }^{ \omega^\frac{1}{2} ( \tau^\frac{1}{2} - \tau^{-\frac{1}{2}} ) } \int_{ \Sph^{n-1} } \Psi | D t |^{-1}_g r^{n-1} |_{ (f, t) = (\omega, s) } ds \text{.}
\end{align}
From \eqref{eql.grad_t_f} and \eqref{eql.coarea_f_0}, we obtain \eqref{eq.coarea_f}.
\end{proof}

We define analogous notations for level sets of $h$.
Any $\mc{H}^\tau$ can be foliated as
\begin{align}
\label{eq.h_comp} \mc{H}^\tau \simeq (0, \infty) \times \Sph^{n-1} \text{,}
\end{align}
where the first component now represents the $r$-coordinate, while the second is again the spherical value.
Moreover, by \eqref{eq.spheres_fh}, the same correspondence yields
\begin{align}
\label{eq.h_comp_cutoff} \mc{H}^\tau_{\rho, \omega} \simeq ( \omega^\frac{1}{2} ( \sigma^\frac{1}{2} - \sigma^{-\frac{1}{2}} ), \omega^\frac{1}{2} ( \tau^\frac{1}{2} - \tau^{-\frac{1}{2}} ) ) \times \Sph^{n-1} \text{.}
\end{align}
We also wish to split integrals over $\mc{H}^\tau$ accordingly.
Thus, for any $\Psi \in \mc{C}^\infty ( \mc{D} )$, we define, in a manner analogous to \eqref{eq.f_int}, the shorthands
\begin{align}
\label{eq.h_int} \int_{ f = \rho }^{ f = \omega } \int_{ \Sph^{n-1} } \Psi |_{ h = \tau } dr &:= \int_{ \rho^\frac{1}{2} ( \tau^\frac{1}{2} + \tau^{-\frac{1}{2}} ) }^{ \omega^\frac{1}{2} ( \tau^\frac{1}{2} + \tau^{-\frac{1}{2}} ) } \left[ \int_{ \Sph^{n-1} } \Psi |_{ (h, r) = (\tau, s) } \right] ds \text{,} \\
\notag \int_0^\infty \int_{ \Sph^{n-1} } \Psi |_{ h = \tau } dr &:= \int_0^\infty \left[ \int_{ \Sph^{n-1} } \Psi |_{ (h, r) = (\tau, s) } \right] ds \text{,}
\end{align}
where the second equation is defined only when $\Psi$ is sufficiently integrable.

\begin{proposition} \label{thm.coarea_h}
For any $\Psi \in \mc{C}^\infty ( \mc{D} )$, we have the identity
\begin{align}
\label{eq.coarea_h} \int_{ \mc{H}^\tau_{\rho, \omega} } \Psi &= 2 \int_{ f = \rho }^{ f = \omega } \int_{ \Sph^{n-1} } f^\frac{1}{2} \Psi r^{n-2} |_{ h = \tau } dr \text{.}
\end{align}
\end{proposition}

\begin{proof}
The proof is analogous to that of \eqref{eq.coarea_f}.
Let $D r$ denote the gradient of $r$ on the level sets of $h$.
By \eqref{eq.h_grad}, both $D r$ and $N$, as defined in \eqref{eq.exhaustion_normal_f}, are tangent to the level sets of $h$ and are normal to the level spheres of $(t, r)$.
Thus,
\begin{align}
\label{eql.grad_r_h} | D r |_g &= | g ( N, D r ) | = \frac{1}{2} \left( \sqrt{ \frac{-u}{v} } + \sqrt{ \frac{v}{-u} } \right) = \frac{1}{2} f^{-\frac{1}{2}} r \text{.}
\end{align}
The result now follows from the coarea formula, \eqref{eq.spheres_fh}, and \eqref{eql.grad_r_h}.
\end{proof}

\subsubsection{The Conformal Inversion} \label{sec:uc_boundary_inversion}

While \eqref{eq.coarea_f} provides a formula for integrals over level sets of $f$, it is 
poorly adapted for the limit $f \nearrow \infty$ toward null infinity.
To handle this limit precisely, we make use of the standard conformal inversion of Minkowski spacetime to identify null infinity with the null cone about the origin.

Consider the conformally inverted metric,
\begin{align}
\label{eq.met_inversion} \gi := f^{-2} g = -4 f^{-2} du dv + f^{-2} r^2 \mathring{\gamma} \text{.}
\end{align}
Furthermore, define the inverted null coordinates
\begin{align}
\label{eq.uv_inv} \ui := - \frac{1}{v} = f^{-1} u \text{,} \qquad \vi := - \frac{1}{u} = f^{-1} v \text{,}
\end{align}
as well as the inverted time and radial parameters
\begin{align}
\label{eq.tr_inv} \ti := \vi + \ui = f^{-1} t \text{,} \qquad \ri := \vi - \ui = f^{-1} r \text{,}
\end{align}
The inverted counterparts of the hyperbolic functions $f$ and $h$ are simply
\begin{align}
\label{eq.fh_inv} \ffi := - \ui \vi = f^{-1} \text{,} \qquad \hhi := - \frac{\vi}{\ui} = h \text{.}
\end{align}
From \eqref{eq.met_inversion}-\eqref{eq.tr_inv}, we see that in terms of the inverted null coordinates,
\begin{equation}
\label{eq.met_inv} \gi = -4 d \ui d \vi + \ri^2 \mathring{\gamma} \text{.}
\end{equation}
In other words, $\gi$, in these new coordinates, is once again the Minkowski metric.

Note that $\mc{D}$ has an identical characterization in this inverted setting:
\begin{align}
\label{eq.region_inv} \mc{D} = \{ Q\in \R^{n+1} \mid \ui (Q) < 0 \text{, } \vi (Q) > 0 \} \text{.}
\end{align}
Moreover, by \eqref{eq.fh_inv}, the outer half of null infinity in the physical setting, $f \nearrow \infty$, corresponds to the null cone about the origin in the inverted setting, $\ffi \searrow 0$.
Thus, this inversion provides a useful tool for discussing behaviors near or at infinity.

We now derive a ``dual" coarea formula for level sets of $f$:

\begin{proposition} \label{thm.coarea_fh_inv}
For any $\Psi \in \mc{C}^\infty ( \mc{D} )$, we have
\begin{align}
\label{eq.coarea_f_inv} \int_{ \mc{F}^{\sigma, \tau}_\omega } \Psi &= 2 \omega^{ n - \frac{1}{2} } \int_{ \hhi = \sigma }^{ \hhi = \tau } \int_{ \Sph^{n-1} } \Psi \ri^{n-2} |_{ \ffi = \omega^{-1} } d\ti \text{.}
\end{align}
\end{proposition}

\begin{proof}
Recalling the relations \eqref{eq.fh_inv}, and observing that the volume forms induced by $g$ and $\gi$ on $\mc{F}^{\sigma, \tau}_\omega$ differ by a factor of $f^n$, we obtain that
\begin{align}
\label{eql.coarea_f_inv_1} \int_{ \mc{F}^{\sigma, \tau}_\omega } \Psi = \omega^n \int_{ ( \bar{\mc{F}}^{\sigma, \tau}_{ 1/\omega }, \gi ) } \Psi \text{,}
\end{align}
where right-hand integral is with respect to the $\gi$-volume form.
Applying \eqref{eq.coarea_f} to the right-hand side of \eqref{eql.coarea_f_inv_1} with respect to $\gi$ yields \eqref{eq.coarea_f_inv}.
\end{proof}

\begin{remark}
In fact, one can obtain corresponding unique continuation results from the null cone by applying the main theorems, such as Theorem \ref{thm.uc_finite}, to the inverted Minkowski spacetime $(\R^{1+n}, \gi)$, and then expressing all objects back in terms of $g$.
\end{remark}

\subsubsection{The Boundary Limit Lemmas}

We now apply the preceding coarea formulas to obtain the desired boundary limits.
We begin with the level sets of $h$.

\begin{lemma} \label{thm.coarea_h_zero}
Fix $\delta > 0$ and $0 < \rho < \omega$, and suppose $\Psi \in \mc{C}^0 (\mc{D})$ satisfies
\begin{align}
\label{eq.coarea_h_zero_ass} \sup_{ \mc{D} } ( r^{n - 1 + \delta} | \Psi | ) < \infty \text{.}
\end{align}
Then, the following limits hold:
\begin{align}
\label{eq.coarea_h_zero} \lim_{ \tau \nearrow \infty } \int_{ \mc{H}^\tau_{\rho, \omega} } | \Psi | = \lim_{ \sigma \searrow 0 } \int_{ \mc{H}^\sigma_{\rho, \omega} } | \Psi | = 0 \text{.}
\end{align}
\end{lemma}

\begin{proof}
First, for $\tau \gg 1$, we have from Lemma \ref{thm.spheres_fh} that
\begin{align}
\label{eql.H_est_large} t |_{ \mc{H}^\tau_{\rho, \omega} } \simeq_{\rho, \omega} \tau^\frac{1}{2} \text{,} \qquad r |_{ \mc{H}^\tau_{\rho, \omega} } \simeq_{\rho, \omega} \tau^\frac{1}{2} \text{.}
\end{align}
Applying \eqref{eq.coarea_h}, \eqref{eql.H_est_large}, and the boundedness of $f$ on $\mc{H}^\tau_{\rho, \omega}$, we see that
\begin{align}
\label{eql.coarea_h_est_1} \int_{ \mc{H}^\tau_{\rho, \omega} } | \Psi | &\lesssim \tau^{-\frac{1 + \delta}{2}} \int_{ f = \rho }^{ f = \omega } \int_{ \Sph^{n-1} } ( \tau^\frac{n-1+\delta}{2} | \Psi | ) |_{ h = \tau } dr \\
\notag &\lesssim \tau^{-\frac{\delta}{2}} \sup_{ \mc{D} } ( r^{n-1+\delta} | \Psi | ) \text{.}
\end{align}
Letting $\tau \nearrow \infty$ and recalling \eqref{eq.coarea_h_zero_ass} results in the first limit in \eqref{eq.coarea_h_zero}.

For the remaining limit, we note from Lemma \ref{thm.spheres_fh} that for $0 < \sigma \ll 1$,
\begin{align}
\label{eq.H_est_small} -t |_{ \mc{H}^\sigma_{\rho, \omega} } \simeq_{\rho, \omega} \sigma^{-\frac{1}{2}} \text{,} \qquad r |_{ \mc{H}^\sigma_{\rho, \omega} } \simeq_{\rho, \omega} \sigma^{-\frac{1}{2}} \text{.}
\end{align}
If $\sigma$ is small, then \eqref{eq.coarea_h} and \eqref{eq.H_est_small} yields
\begin{align}
\label{eql.coarea_h_est_2} \int_{ \mc{H}^\sigma_{\rho, \omega} } | \Psi | &\lesssim \sigma^{\frac{1 + \delta}{2}} \int_{ f = \rho }^{ f = \omega } \int_{ \Sph^{n-1} } ( \sigma^{-\frac{n-1+\delta}{2}} | \Psi | ) |_{ h = \tau } dr \\
\notag &\lesssim \sigma^\frac{\delta}{2} \sup_{ \mc{D} } ( r^{n-1+\delta} | \Psi | ) \text{,}
\end{align}
which vanishes by \eqref{eq.coarea_h_zero_ass} as $\sigma \searrow 0$.
\end{proof}

Next, we consider level sets of $f$, which we split into  two statements.

\begin{lemma} \label{thm.coarea_f_zero_low}
Suppose $\Psi \in \mc{C}^0 (\mc{D})$ satisfies
\begin{align}
\label{eq.coarea_f_zero_low_ass} \sup_{ \mc{D} } [ (1 + r)^{n - 1 + \delta} | \Psi | ] < \infty \text{.}
\end{align}
Then, for any $\alpha > 0$, the following limit holds:
\begin{align}
\label{eq.coarea_f_zero_low} \lim_{ \rho \searrow 0 } \lim_{ (\sigma, \tau) \rightarrow (0, \infty) } \int_{ \mc{F}^{\sigma, \tau}_\rho } f^{-\frac{1}{2} + \alpha} | \Psi | = 0 \text{.}
\end{align}
\end{lemma}

\begin{proof}
Applying \eqref{eq.coarea_f} and \eqref{eq.coarea_f_zero_low_ass}, we see that
\begin{align}
\label{eql.coarea_f_est_1} \int_{ \mc{F}_\rho^{\sigma, \tau} } f^{-\frac{1}{2} + \alpha} | \Psi | &\lesssim \rho^\alpha \int_{ h = \sigma }^{ h = \tau } \int_{ \Sph^{n-1} } | \Psi | r^{n-2} |_{ f = \omega } dt \\
\notag &\lesssim \rho^\alpha \int_{-\infty}^\infty (1 + r)^{-1 - \delta} dt \text{.}
\end{align}
Since $|t| < r$ on $\mc{D}$, then \eqref{eq.coarea_f_zero_low} follows by letting $(\sigma, \tau) \rightarrow (0, \infty)$ and $\rho \searrow 0$.
\end{proof}

\begin{lemma} \label{thm.coarea_f_zero_high}
Suppose $\Psi \in \mc{C}^0 (\mc{D})$ satisfies
\begin{align}
\label{eq.coarea_f_zero_high_ass} \sup_{ \mc{D} } [ (r + f)^{n - 1 + \delta} | \Psi | ] < \infty \text{.}
\end{align}
Then, for any $0 < \beta < \delta$, the following limit holds:
\begin{align}
\label{eq.coarea_f_zero_high} \lim_{ \omega \nearrow \infty } \lim_{ (\sigma, \tau) \rightarrow (0, \infty) } \int_{ \mc{F}_\omega^{\sigma, \tau} } f^{-\frac{1}{2} + \beta} | \Psi | = 0 \text{.}
\end{align}
\end{lemma}

\begin{proof}
The main idea is to convert to the inverted setting, in which the estimate becomes analogous to Lemma \ref{thm.coarea_f_zero_low}.
From \eqref{eq.coarea_f_inv}, we obtain
\begin{align}
\label{eql.coarea_f_est_2} \int_{ \mc{F}_\omega^{\sigma, \tau} } f^{-\frac{1}{2} + \beta} | \Psi | &\lesssim \omega^{ n - 1 + \beta } \int_{ \hhi = \sigma }^{ \hhi = \tau } \int_{ \Sph^{n-1} } | \Psi | \ri^{n-2} |_{ \ffi = \omega^{-1} } d\ti \\
\notag &\lesssim \omega^{n -1 + \beta} \int_{-\infty}^\infty 
(r + f)^{- (n - 1 - \delta)} \ri^{n-2} d\ti \text{.}
\end{align}
Recalling now \eqref{eq.tr_inv} and \eqref{eq.fh_inv} yields
\begin{align}
\label{eql.coarea_f_est_3} \int_{ \mc{F}_\omega^{\sigma, \tau} } 
f^{-\frac{1}{2} + \beta} | \Psi | &\lesssim \omega^{n - 1 + \beta}
 \int_{-\infty}^\infty 
(r + f)^{ -(n - 1 + \delta) } \ri^{n-2} d\ti \\
\notag &\lesssim \omega^{n - 1 + \beta} \int_{-\infty}^\infty 
\omega^{ - (n - 1 + \delta) } (1 + \ri)^{ - (n - 1 + \delta) } \ri^{n-2} d\ti \\
\notag &\lesssim \omega^{\beta - \delta} \int_{-\infty}^\infty 
(1 + \ri)^{-1 - \delta} d \ti \text{.}
\end{align}
Since $| \ti | < \ri$, then letting $(\sigma, \tau) \rightarrow (0, \infty)$ and 
$\omega \nearrow \infty$ results in \eqref{eq.coarea_f_zero_high}.
\end{proof}

\subsection{Proof of Theorem \ref{thm.uc_finite}}

Let $p$ be as in the theorem statement, and let
\begin{align}
\label{eql.uc_abp} a = \frac{\beta}{4} + \frac{p}{4} \text{,} \qquad b = \frac{1}{16} \min ( \beta - p, 8 p ) \text{.}
\end{align}
Note in particular that $2 a$ lies between $p$ and $\beta$, and that $a, b, p$ satisfy the conditions \eqref{eq.abp_lh}.
In addition, we fix arbitrary
\begin{align*}
0 < \rho < 1 < \omega < \infty \text{,} \qquad 0 < \sigma < \tau < \infty \text{.}
\end{align*}
The idea is apply Theorem \ref{thm.carleman_lh} to the domain $\mc{D}^{\sigma, \tau}_{\rho, \omega}$.
To split into $f < 1$ and $f > 1$ regions, we partition the above domain into two parts, $\mc{D}^{\sigma, \tau}_{\rho, 1}$ and $\mc{D}^{\sigma, \tau}_{1, \omega}$.

First, applying \eqref{eq.carleman_l} and then \eqref{eq.divg_fh} to $\mc{D}^{\sigma, \tau}_{\rho, 1}$ yields
\begin{align}
\label{eql.uc_finite_100} C b^2 p \int\limits_{ \mc{D}^{\sigma, \tau}_{\rho, 1} } f^{ 2 (a - b) } f^{ p - 1 } \phi^2 &\leq K a^{-1} \int\limits_{ \mc{D}^{\sigma, \tau}_{\rho, 1} } f^{ 2 (a - b) } f | \Box \phi |^2 + \int\limits_{ \mc{F}^{\sigma, \tau}_1 } f^{-\frac{1}{2}} P^-_\beta \nabla^\beta f \\
\notag &\qquad - \int\limits_{ \mc{F}^{\sigma, \tau}_\rho } f^{-\frac{1}{2}} P^-_\beta \nabla^\beta f + \int\limits_{ \mc{H}^\tau_{\rho, 1} } u^2 f^{-\frac{1}{2}} P^-_\beta \nabla^\beta h \\
\notag &\qquad - \int\limits_{ \mc{H}^\sigma_{\rho, 1} } u^2 f^{-\frac{1}{2}} P^-_\beta \nabla^\beta h \text{.}
\end{align}
Observe that the assumption \eqref{eq.uc_potential_decay} for $\mc{V}$ and \eqref{eql.uc_abp} imply that
\begin{align}
\label{eql.uc_finite_110} | \Box \phi |^2 \leq | \mc{V} |^2 \phi^2 \leq B^2 b^2 p^2 f^{-2 + p} \phi^2 \text{,}
\end{align}
whenever $f < 1$.
Applying \eqref{eql.uc_finite_110} to \eqref{eql.uc_finite_100} and noting that $p < 2a$, it follows that if $B$ is sufficiently small, then the first term on the right-hand side of \eqref{eql.uc_finite_100} can be absorbed into the left-hand side, yielding, for some other $C > 0$,
\begin{align}
\label{eql.uc_finite_200} C b^2 p \int\limits_{ \mc{D}^{\sigma, \tau}_{\rho, 1} } f^{ 2 (a - b) } f^{ p - 1 } \phi^2 &\leq - \int\limits_{ \mc{F}^{\sigma, \tau}_\rho } f^{-\frac{1}{2}} P^-_\beta \nabla^\beta f + \int\limits_{ \mc{H}^\tau_{\rho, 1} } u^2 f^{-\frac{1}{2}} P^-_\beta \nabla^\beta h \\
\notag &\qquad - \int\limits_{ \mc{H}^\sigma_{\rho, 1} } u^2 f^{-\frac{1}{2}} P^-_\beta \nabla^\beta h + \int\limits_{ \mc{F}^{\sigma, \tau}_1 } f^{-\frac{1}{2}} P^-_\beta \nabla^\beta f \text{.}
\end{align}

In a similar manner, we apply \eqref{eq.carleman_h} and \eqref{eq.divg_fh} to $\mc{D}^{\sigma, \tau}_{1, \omega}$.
In this region, the assumption \eqref{eq.uc_potential_decay} for $\mc{V}$ and \eqref{eql.uc_abp} imply the estimate
\begin{align}
\label{eql.uc_finite_210} | \Box \phi |^2 \leq | \mc{V} |^2 \phi^2 \leq B^2 b^2 p^2 \cdot f^{-2 - p} 
\phi^2 \text{,}
\end{align}
so that for sufficiently small $B$, we have
\begin{align}
\label{eql.uc_finite_300} C b^2 p \int\limits_{ \mc{D}^{\sigma, \tau}_{1, \omega} } f^{ 2 (a + b) } f^{ -p - 1 } \phi^2 &\leq \int\limits_{ \mc{F}^{\sigma, \tau}_\omega } f^{-\frac{1}{2}} P^+_\beta \nabla^\beta f + \int\limits_{ \mc{H}^\tau_{1, \omega} } u^2 f^{-\frac{1}{2}} P^+_\beta \nabla^\beta h \\
\notag &\qquad - \int\limits_{ \mc{H}^\sigma_{1, \omega} } u^2 f^{-\frac{1}{2}} P^+_\beta \nabla^\beta h - \int\limits_{ \mc{F}^{\sigma, \tau}_1 } f^{-\frac{1}{2}} P^+_\beta \nabla^\beta f \text{.}
\end{align}

Next, we sum \eqref{eql.uc_finite_200} and \eqref{eql.uc_finite_300}.
By \eqref{eq.carleman_lh_match}, the two integrals over $\mc{F}_1^{\sigma, \tau}$ cancel:
\begin{align}
\label{eql.uc_finite_400} &C b^2 p \int\limits_{ \mc{D}^{\sigma, \tau}_{\rho, 1} } f^{ 2 (a - b) } f^{ p - 1 } \phi^2 + C b p^2 \int\limits_{ \mc{D}^{\sigma, \tau}_{1, \omega} } f^{ 2 (a + b) } f^{ -p - 1 } \phi^2 \\
\notag &\quad \leq \int\limits_{ \mc{F}^{\sigma, \tau}_\omega } f^{-\frac{1}{2}} P^+_\beta \nabla^\beta f - \int\limits_{ \mc{F}^{\sigma, \tau}_\rho } f^{-\frac{1}{2}} P^-_\beta \nabla^\beta f + \int\limits_{ \mc{H}^\tau_{\rho, 1} } u^2 f^{-\frac{1}{2}} P^-_\beta \nabla^\beta h \\
\notag &\quad \qquad + \int\limits_{ \mc{H}^\tau_{1, \omega} } u^2 f^{-\frac{1}{2}} P^+_\beta \nabla^\beta h - \int\limits_{ \mc{H}^\sigma_{\rho, 1} } u^2 f^{-\frac{1}{2}} P^-_\beta \nabla^\beta h - \int\limits_{ \mc{H}^\sigma_{1, \omega} } u^2 f^{-\frac{1}{2}} P^+_\beta \nabla^\beta h \\
\notag &\quad = I_1 + I_2 + J_1 + J_2 + J_3 + J_4 \text{.}
\end{align}
It remains to show that each of the terms on the right vanishes in the limit.

First, for $J_1$, we apply \eqref{eq.boundary_l_est}, along with the fact that $f$ is bounded from both above and below on $\mc{H}^\tau_{\rho, 1}$, in order to obtain
\begin{align}
\label{eql.uc_finite_411} | J_1 | &\lesssim \int_{ \mc{H}^\tau_{\rho, 1} } [ ( u \cdot \partial_u \phi )^2 + ( v \cdot \partial_v \phi )^2 + \phi^2 ] \text{.}
\end{align}
An analogous application of \eqref{eq.boundary_h_est} yields
\begin{align}
\label{eql.uc_finite_412} | J_2 | &\lesssim \int_{ \mc{H}^\sigma_{\rho, 1} } [ ( u \cdot \partial_u \phi )^2 + ( v \cdot \partial_v \phi )^2 + \phi^2 ] \text{.}
\end{align}
Recalling our assumption \eqref{eq.wave_decay_ex} for $\phi$ and applying Lemma \ref{thm.coarea_h_zero} yields
\begin{align}
\label{eql.uc_finite_420} \lim_{\tau \nearrow \infty} J_1 = \lim_{\sigma \searrow 0} J_2 = 0 \text{.}
\end{align}
By the same arguments, we also obtain
\begin{align}
\label{eql.uc_finite_430} \lim_{\tau \nearrow \infty} J_3 = \lim_{\sigma \searrow 0} J_4 = 0 \text{.}
\end{align}

For the remaining terms $I_1$ and $I_2$, we take limits first as $(\tau, \sigma) \rightarrow (\infty, 0)$, and then as $(\omega, \rho) \rightarrow (\infty, 0)$.
First, using \eqref{eq.boundary_l_est}, we obtain
\begin{align}
\label{eql.uc_finite_441} | I_2 | \lesssim \int_{ \mc{F}^{\sigma, \tau}_\rho } f^{-\frac{1}{2} + 2 (a - b) } ( f \cdot | \nasla \phi |^2 + \phi^2 ) 
\end{align}
Since $a - b > 0$, then Lemma \ref{thm.coarea_f_zero_low} and the decay assumption \eqref{eq.wave_decay_ex} imply
\begin{align}
\label{eql.uc_finite_450} \lim_{\rho \searrow 0} \lim_{ (\sigma, \tau) \rightarrow (0, \infty) } I_2 = 0 \text{.}
\end{align}
A similar application of \eqref{eq.boundary_h_est} yields
\begin{align}
\label{eql.uc_finite_461} | I_1 | \lesssim \int_{ \mc{F}^{\sigma, \tau}_\omega } f^{-\frac{1}{2} + 2 (a + b) } [ ( u \cdot \partial_u \phi )^2 + ( v \cdot \partial_v \phi )^2 + \phi^2 ]  
\end{align}
Since $2 (a + b) < \beta$, then Lemma \ref{thm.coarea_f_zero_high} and \eqref{eq.wave_decay_ex} yield
\begin{align}
\label{eql.uc_finite_470} \lim_{\omega \nearrow \infty} \lim_{ (\sigma, \tau) \rightarrow (0, \infty) } I_1 = 0 \text{.}
\end{align}

Finally, combining \eqref{eql.uc_finite_400} with the limits \eqref{eql.uc_finite_420}, \eqref{eql.uc_finite_430}, \eqref{eql.uc_finite_450}, \eqref{eql.uc_finite_470} and then applying the monotone convergence theorem, we see that
\begin{align}
\label{eql.uc_finite_500} \int_{ \mc{D} } \alpha \cdot \phi^2 = 0 \text{,}
\end{align}
for some strictly positive function $\alpha$ on $\mc{D}$.
It follows that $\phi$ vanishes everywhere on $\mc{D}$, completing the proof of Theorem \ref{thm.uc_finite}.

\subsection{Proofs of Theorems \ref{thm.uc_focusing} and \ref{thm.uc_defocusing}}

The idea is similar to before, except we apply Theorem \ref{thm.carleman_nl} instead of Theorem \ref{thm.carleman_lh}.
In particular, we let $\Omega = \mc{D}_{\rho, \omega}^{\sigma, \tau}$ (there is no need to partition into two regions in this case), and fix
\begin{align}
\label{eql.uc_abp_nonlinear} 0 < a < \frac{\delta}{2} \text{,}
\end{align}
whose precise value is to be determined later.
Let $V$ be as in the statement of Theorem \ref{thm.uc_focusing} or \ref{thm.uc_defocusing}.
Choosing the sign depending on situation (``$+$" for Theorem \ref{thm.uc_focusing}, or ``$-$" for Theorem \ref{thm.uc_defocusing}), we have that $\Box_V^\pm \phi \equiv 0$.

As a result, applying \eqref{eq.carleman_nl}, \eqref{eq.divg_fh}, and \eqref{eq.boundary_nl_est}, we obtain
\begin{align}
\label{eql.uc_nonlinear_100} \int_{ \mc{D}^{\sigma, \tau}_{\rho, \omega} } f^{2a} V (\pm \Gamma_V) | \phi |^{p + 1} &\leq L \int_{ \mc{F}^{\sigma, \tau}_\omega } f^{2 a} f^{-\frac{1}{2}} [ ( u \cdot \partial_u \phi )^2 + ( v \cdot \partial_v \phi )^2 + \phi^2 ] \\
\notag &\quad \qquad + L \int_{ \mc{F}^{\sigma, \tau}_\rho } f^{2 a} f^{-\frac{1}{2}} [ f \cdot | \nasla \phi |^2 + \phi^2 ] \\
\notag &\quad \qquad + L \int_{ \mc{H}^\tau_{\rho, \omega} } f^{2 a} f^{-\frac{1}{2}} [ ( v \cdot \partial_v \phi )^2 + \phi^2 ] \\
\notag &\quad \qquad + L \int_{ \mc{H}^\sigma_{\rho, \omega} } f^{2 a} f^{-\frac{1}{2}} [ ( u \cdot \partial_u \phi )^2 + \phi^2 ] \\
\notag &\quad \qquad \pm \int_{ \mc{F}^{\sigma, \tau}_\omega } f^{2 a} f^{-\frac{1}{2}} ( f V \cdot | \phi |^{p+1} ) \\
\notag &\quad \qquad \mp \int_{ \mc{F}^{\sigma, \tau}_\rho } f^{2 a} f^{-\frac{1}{2}} ( f V \cdot | \phi |^{p+1} ) \\
\notag &\quad = I_1 + I_2 + J_1 + J_2 + Z_1 + Z_2 \text{,}
\end{align}
for some $L > 0$.
Moreover, by the same arguments as in the proof of Theorem \ref{thm.uc_finite},
\begin{align}
\label{eql.uc_nonlinear_110} \lim_{ (\rho, \omega) \rightarrow (0, \infty) } \lim_{ (\sigma, \tau) \rightarrow (0, \infty) } ( I_1 + I_2 + J_1 + J_2 ) = 0 \text{.}
\end{align}

Suppose first that we are in the defocusing case of Theorem \ref{thm.uc_defocusing}.
Then, $Z_1$ is negative and hence can be discarded, and we need only consider $Z_2$.
Since $V$ is uniformly bounded, and since $\phi$ satisfies \eqref{eq.wave_decay_ex}, then applying \eqref{eq.coarea_f_zero_low} yields
\begin{align}
\label{eql.uc_nonlinear_2neg} \lim_{ \rho \searrow 0 } \lim_{ (\sigma, \tau) \rightarrow (0, \infty) } | Z_2 | &\lesssim \lim_{ \rho \searrow 0 } \lim_{ (\sigma, \tau) \rightarrow (0, \infty) } \int_{ \mc{F}^{\sigma, \tau}_\rho } f^{2 a} f^\frac{1}{2} \cdot | \phi |^{p+1} = 0 \text{.}
\end{align}

On the other hand, in the focusing case of Theorem \ref{thm.uc_focusing}, we can discard $Z_2$ due to sign and consider only $Z_1$.
Since $f \nearrow \infty$, the factor $f$ in $Z_1$ now becomes large, hence the extra decay for $\phi$ in \eqref{eq.wave_decay_focusing} is invoked to show the vanishing of this limit.
Indeed, applying the decay assumptions \eqref{eq.wave_decay_ex}, \eqref{eq.wave_decay_focusing_ex} along with \eqref{eq.coarea_f_zero_high}, we see that
\begin{align}
\label{eql.uc_nonlinear_2pos} \lim_{ \omega \nearrow \infty } \lim_{ (\sigma, \tau) \rightarrow (0, \infty) } | Z_1 | &\lesssim \lim_{ \omega \nearrow \infty } \lim_{ (\sigma, \tau) \rightarrow (0, \infty) } \int_{ \mc{F}^{\sigma, \tau}_\omega } f^{2 a} f^\frac{1}{2} V \cdot | \phi |^{p+1} = 0 \text{.}
\end{align}
Consequently, in both cases, we can do away with all the boundary terms:
\begin{align}
\label{eql.uc_nonlinear_200} \lim_{ (\rho, \omega) \rightarrow (0, \infty) } \lim_{ (\sigma, \tau) \rightarrow (0, \infty) } ( I_1 + I_2 + J_1 + J_2 + Z_1 + Z_2 ) = 0 \text{,}
\end{align}

It remains only to examine the left-hand side of \eqref{eql.uc_nonlinear_100}.
In the defocusing case, the monotonicity assumption \eqref{eq.uc_defocusing_mono} and \eqref{eq.f_grad} imply that
\begin{align}
\label{eql.uc_nonlinear_3neg} \grad f ( \log V ) < \frac{n - 1 + 4a}{4} \left( p - 1 - \frac{4}{n - 1 + 4a} \right) \text{,}
\end{align}
so that $- \Gamma_V$, as defined in \eqref{eq.bulk_gamma}, is strictly positive.
Similarly, in the focusing case, we can further shrink $a > 0$ (depending on $\mu$ in \eqref{eq.uc_focusing_mono}) to guarantee that
\begin{align}
\label{eql.uc_nonlinear_3pos} \grad f ( \log V ) > - \frac{n - 1 + 4a}{4} \left( 1 + \frac{4}{n - 1 + 4a} - p \right) \text{,}
\end{align}
so that $\Gamma_V$ is strictly positive.
Consequently, in both cases, the factor $\pm f^{2a} V \cdot \Gamma_V$ on the left-hand side of \eqref{eql.uc_nonlinear_100} is strictly positive.

Finally, we combine the above positivity of $\pm \Gamma_V$ with \eqref{eql.uc_nonlinear_100} and \eqref{eql.uc_nonlinear_200}, and we apply the monotone convergence theorem.
This yields
\begin{align}
\label{eql.uc_nonlinear_300} \int_\mc{D} \alpha | \phi |^{p + 1} = 0 \text{,}
\end{align}
where $\alpha$ is a strictly positive function on $\mc{D}$.
It follows that $\phi$ vanishes everywhere on $\mc{D}$, which completes the proofs of both Theorems \ref{thm.uc_focusing} and \ref{thm.uc_defocusing}.

\subsection{Proof of Proposition \ref{thm.uc_opt}} \label{sec:uc_opt}

We claim there exists $\psi \in \mc{C}^2 (\R^n)$ satisfying
\begin{align}
\label{eq.opt_equation} ( \Delta + \mc{U} ) \psi = 0 \text{,}
\end{align}
where $\mc{U} \in \mc{C}^\infty (\R^n)$ is compactly supported, and where
\begin{align}
\label{eq.opt_decay} | \psi (t, x) | \lesssim ( 1 + |x| )^{-k} \text{,}
\end{align}
where $k > 0$ is fixed but can be made arbitrarily large.
Assuming this $\psi$, taking
\begin{align*}
\phi (t, x) = \psi (x) \text{,} \qquad \mc{V} (t, x) = \mc{U} (x)
\end{align*}
results in the desired counterexample.
Thus, it remains only to construct $\psi$.

\subsubsection{Construction of $\psi$}

Let $P \in \mc{C}^\infty ( \Sph^{n-1} )$ denote a spherical harmonic, satisying
\begin{align*}
\Delta_{ \Sph^{n-1} } P = - a P \text{,} \qquad a > 0 \text{,}
\end{align*}
and let $\beta \in \mc{C}^\infty (0, \infty)$ be an everywhere positive function satisfying
\begin{align*}
\beta (r) := \begin{cases} r^{q_+} & r < 1 \text{,} \\ r^{q_-} & r > 2 \text{,} \end{cases} \qquad q_\pm := \frac{ - (n - 2) \pm \sqrt{ (n - 2)^2 + 4 a } }{2} \text{.}
\end{align*}
Observe in particular that $q_- < 0 < q_+$, and that $| q_\pm |$ can be made arbitrarily large by making $a$ arbitrarily large.
Thus, if we define, in polar coordinates, the function
\begin{align*}
\psi (r, \omega) := \beta (r) \cdot P (\omega) \text{,}
\end{align*}
then for large enough $a$, this defines a $\mc{C}^2$-function on $\R^n$ satisfying \eqref{eq.opt_decay}.

Finally, we define
\begin{align*}
\mc{U} (r, \omega) := \frac{ - \Delta \psi (r, \omega) }{ \psi (r, \omega) } = \frac{ - \beta'' (r) - (n - 1) r^{-1} \cdot \beta' (r) + a r^{-2} \cdot \beta (r) }{ \beta (r) } \text{,}
\end{align*}
which extends to a smooth function on $\R^n \setminus \{ 0 \}$.
Furthermore, the chosen exponents $q_\pm$ ensure that $\Delta \psi (r, \omega)$ vanishes whenever $r < 1$ or $r > 2$.
As a result, $\mc{U}$ is actually smooth on all of $\R^n$ and has compact support, and \eqref{eq.opt_equation} is satisfied.

\subsection*{Acknowledgments}

We are grateful to Sergiu Klainerman and Alex Ionescu for many helpful conversations, and we thank them for generously sharing their results in \cite{io_kl:private}.
The first author was partially supported by NSERC grants 488916 and 489103, and 
the ERA grant 496588.

\raggedright
\bibliographystyle{amsplain}
\bibliography{bib}

\providecommand{\bysame}{\leavevmode\hbox to3em{\hrulefill}\thinspace}
\providecommand{\MR}{\relax\ifhmode\unskip\space\fi MR }
\providecommand{\MRhref}[2]{%
  \href{http://www.ams.org/mathscinet-getitem?mr=#1}{#2}
}
\providecommand{\href}[2]{#2}
\begin{thebibliography}{10}

\bibitem{agmon:lower}
S.~Agmon, \emph{Lower bounds for solutions of {Schr\"odinger} equations}, J.
  Analyse Math. \textbf{23} (1970), 1--25.

\bibitem{alex_schl_shao:uc_inf}
S.~Alexakis, V.~Schlue, and A.~Shao, \emph{Unique continuation from infinity
  for linear waves}, arXiv:1312.1989, 2013.

\bibitem{alex_shao:uc_nlw}
S.~Alexakis and A.~Shao, \emph{On the profile of energy concentration at
  blow-up points for focusing sub-conformal nonlinear waves.}, arXiv:1412.1537,
  2014.

\bibitem{alin_baou:non_unique}
S.~Alinhac and M.~S. Baouendi, \emph{A non uniqueness result for operators of
  principal type}, Math. Z. \textbf{220} (1995), no.~1, 561--568.

\bibitem{bic_scho_tod:time_per_sc:1}
J.~{Bi\v{c}\'ak}, M.~Scholtz, and P.~Tod, \emph{On asymptotically flat
  solutions of {Einstein's} equations periodic in time: {I.} {Vacuum} and
  electrovacuum solutions}, Classical Quant. Grav. \textbf{27} (2010), no.~5,
  055007.

\bibitem{bic_scho_tod:time_per_sc:2}
\bysame, \emph{On asymptotically flat solutions of {Einstein's} equations
  periodic in time: {II.} {Spacetimes} with scalar-field sources}, Classical
  Quant. Grav. \textbf{27} (2010), no.~17, 175011.

\bibitem{chr:null_cond}
D.~Christodoulou, \emph{Global solutions of nonlinear hyperbolic equations for
  small initial data}, Comm. Pure App. Math. \textbf{39} (1986), 267--282.

\bibitem{fr:rad_scatt}
F.~G. Friedlander, \emph{Radiation fields and hyperbolic scattering theory},
  Math. Proc. Camb. Phil. Soc. \textbf{88} (1980), no.~3, 483--515.

\bibitem{fr:notes_wave}
\bysame, \emph{Notes on the wave equation on asymptotically {Euclidean}
  manifolds}, J. Funct. Anal. \textbf{184} (2001), no.~1, 1--18.

\bibitem{haw_el:gr}
S.~F. Hawking and G.~F.~R. Ellis, \emph{The large scale structure of
  space-time}, Cambridge University Press, 1975.

\bibitem{hor:lpdo2}
L.~H{\"o}rmander, \emph{The analysis of linear partial differential operators
  {II:} {Differential} operators with constant coefficients}, Springer-Verlag,
  1985.

\bibitem{hor:lpdo4}
\bysame, \emph{The analysis of linear partial differential operators {IV:}
  {Fourier} integral operators}, Springer-Verlag, 1985.

\bibitem{io_kl:unique_ip}
A.~Ionescu and S.~Klainerman, \emph{Uniqueness results for ill-posed
  characteristic problems in curved space-times}, Commun. Math. Phys.
  \textbf{285} (2009), no.~3, 873--900.

\bibitem{io_kl:private}
A.~D. Ionescu and S.~Klainerman, Private communication, 2013.

\bibitem{io_kl:killing}
\bysame, \emph{On the local extension of {Killing} vector-fields in {Ricci}
  flat manifolds}, J. Amer. Math. Soc. \textbf{26} (2013), 563--593.

\bibitem{kato:growth}
T.~Kato, \emph{Growth properties of solutions of the reduced wave equation with
  a variable coefficient}, Comm. Pure Appl. Math. \textbf{12} (1959), 403--425.

\bibitem{ken_ruiz_sog:sobolev_unique}
C.~E. Kenig, A.~Ruiz, and C.~D. Sogge, \emph{Uniform {Sobolev} inequalities and
  unique continuation for second order constant coefficient differential
  operators}, Duke Math. J. \textbf{55} (1987), no.~2, 329--347.

\bibitem{kl:remark_kg}
S.~Klainerman, \emph{Remark on the asymptotic behavior of the {Klein Gordon}
  equation in $\mathbb{R}^{n+1}$}, Comm. Pure App. Math. \textbf{46} (1993),
  no.~2, 137--144.

\bibitem{lax_phil:scatt_theory}
P.~D. Lax and R.~S. Phillips, \emph{Scattering theory}, Academic Press, 1967.

\bibitem{mesh:inf_decay}
V.~Z. Meshkov, \emph{On the possible rate of decay at infinity of solutions of
  second-order partial differential equations}, Math. USSR Sbornik \textbf{72}
  (1992), 343--361.

\bibitem{mesh:diff_ineq}
\bysame, \emph{Weighted differential inequalities and their application for
  estimating the rate of decrease at infinity of solutions of second-order
  elliptic equations}, Proc. Steklov Inst. Math. \textbf{70} (1992), 145--166.

\bibitem{papap:1957}
A.~Papapetrou, \emph{\"{U}ber periodische nichtsingul\"are {L}\"osungen in der
  allgemeinen {R}elativit\"atstheorie}, Ann. Physik (6) \textbf{20} (1957),
  399--411.

\bibitem{papap:1958:1}
\bysame, \emph{\"{U}ber periodische {G}ravitations- und elektromagnetische
  {F}elder in der allgemeinen {R}elativit\"atstheorie}, Ann. Physik (7)
  \textbf{1} (1958), 186--197.

\bibitem{papap:1958:2}
\bysame, \emph{\"{U}ber zeitabh\"angige {L}\"osungen der {F}eldgleichungen der
  allgemeinen {R}elativit\"atstheorie}, Ann. Physik (7) \textbf{2} (1958),
  87--96.

\bibitem{papap:nonradiative}
A.~Papapetrou, \emph{Theorem on nonradiative electromagnetic and gravitational
  fields}, J. Math. Phys \textbf{6} (1965), 1405.

\bibitem{sim:pos_eigen}
B.~Simon, \emph{On positive eigenvalues of one-body {Schr}\"odinger operators},
  Comm. Pure App. Math. \textbf{22} (1967), 531--538.

\bibitem{wald:gr}
R.~Wald, \emph{General relativity}, The University of Chicago Press, 1984.

\end{thebibliography}

\end{document}